\documentclass[onefignum,onetabnum]{siamonline250211}

\usepackage{lipsum}
\usepackage{amsfonts}
\usepackage{graphicx}
\usepackage{epstopdf}
\usepackage{algorithmic}
\ifpdf
  \DeclareGraphicsExtensions{.eps,.pdf,.png,.jpg}
\else
  \DeclareGraphicsExtensions{.eps}
\fi

\newsiamremark{example}{Example}   
\newsiamthm{assumption}{Assumption}

\usepackage{enumitem}

\usepackage{tikz}
\usetikzlibrary{decorations.markings}
\usetikzlibrary{decorations.pathmorphing}
\usetikzlibrary{shapes.geometric, calc,decorations.pathreplacing}
\usetikzlibrary{backgrounds}
\usetikzlibrary{fit}
\usetikzlibrary{arrows}
\tikzset{
	arrow/.pic={\path[tips,every arrow/.try,->,>=#1] (0,0) -- +(.1pt,0);},
	pics/arrow/.default={triangle 90}
}

\newcommand{\wass}{\mathbb{W}}
\newcommand{\norm}[1]{\lVert #1 \rVert}

\newcommand{\rr}{\mathbb{R}}

\newcommand{\id}{\mathrm{id}}
\newcommand{\proj}{\mathrm{Proj}}
\DeclareMathOperator{\Tr}{tr}
\newcommand{\R}{\mathbb{R}}
\newcommand{\E}{\mathbb{E}}

\newcommand{\Hplanes}{\mathcal{H}}

\newcommand{\Prob}{\mathcal{P}}
\newcommand{\abs}[1]{\left\lvert #1 \right\rvert}
\newcommand{\CEF}{\mathcal{F}}

\newcommand{\T}{\mathcal{T}}

\newcommand{\X}{\mathcal{X}}
\newcommand{\Z}{\mathcal{Z}}
\newcommand{\cO}{\mathcal{O}}

\newcommand{\Ent}{\mathrm{Ent}}

\newcommand{\cconv}{c\text{-}\mathrm{conv}}
\newcommand{\cbarconv}{\overline{c}\text{-}\mathrm{conv}}

\DeclareMathOperator{\Cor}{Cor}

\DeclareMathOperator*{\argmin}{arg\,min}

\newsiamremark{remark}{Remark}
\newsiamremark{hypothesis}{Hypothesis}
\crefname{hypothesis}{Hypothesis}{Hypotheses}
\newsiamthm{claim}{Claim}
\newsiamremark{fact}{Fact}
\crefname{fact}{Fact}{Facts}

\headers{On the Wasserstein alignment problem}{S. Pal, B. Sen, and T.-K. L. Wong}

\title{On the Wasserstein alignment problem\thanks{Submitted to the editors DATE.
\funding{SP acknowledges support from NSF grants DMS-2052239, DMS-2134012, DMS-2133244, and PIMS PRN-01 granted to the Kantorovich Initiative. BS would like to acknowledge support from NSF grant DMS-2311062. LW acknowledges support from NSERC Discovery Grants RGPIN-2019-04419, RGPIN-2025-06021.}}}

\author{Soumik Pal\thanks{Department of Mathematics, University of Washington, Seattle, WA, United States
  (\email{soumik@uw.edu}).}
\and Bodhisattva Sen\thanks{Department of Statistics, Columbia University, 
New York City, NY, United States 
  (\email{bodhi@stat.columbia.edu}).}
\and Ting-Kam Leonard Wong\thanks{Department of Statistical Sciences, University of Toronto, Toronto, ON, Canada  (\email{tkl.wong@utoronto.ca})}}

\usepackage{amsopn}

\ifpdf
\hypersetup{
  pdftitle={On the Wasserstein alignment problem},
  pdfauthor={S. Pal, B. Sen and T.-K. L. Wong}
}
\fi

\begin{document}

\maketitle

\begin{abstract}
Suppose we are given two metric spaces and a family of continuous transformations from one to the other. Given a probability distribution on each of these two spaces---namely the source and the target measures---the {\it Wasserstein alignment problem} seeks the transformation that minimizes the optimal transport cost between the pushforward of the source distribution and the target distribution, ensuring the closest possible alignment in a probabilistic sense. Examples of interest include two distributions on two Euclidean spaces $\rr^n$ and $\rr^d$, and we want a spatial embedding of the $n$-dimensional source measure in $\rr^d$ that is closest in some Wasserstein metric to the target distribution on $\rr^d$. Similar data alignment problems also commonly arise in shape analysis and computer vision. In this paper, we show that this nonconvex optimal transport projection problem admits a convex Kantorovich-type dual that exploits statistical independence. This allows us to characterize the set of projections and devise a linear programming algorithm. For certain special examples, such as orthogonal transformations on Euclidean spaces of unequal dimensions and the $2$-Wasserstein cost, we characterize the covariance of the optimal projections. Our results also cover the generalization when we penalize each transformation by a function. An example is the inner product Gromov--Wasserstein distance minimization problem which has recently gained popularity. 
\end{abstract}



\section{Introduction} \label{sec:intro}
Aligning two high-dimensional data sets, each represented as a probability distribution on potentially different spaces, is a fundamental problem with diverse applications. This problem arises in various fields, including shape analysis~\cite{Dryden-Mardia-2016, Srivastava-2016}, domain adaptation~\cite{Courty2016,Peyre2019}, unsupervised alignment of word embeddings~\cite{Alvarez2018, Alvarez2019, Grave2019}, and the integration of multi-modal single-cell sequencing data~\cite{Bunne2024, Demetci2022}, among others. Although the precise definition of alignment varies across studies, most formulations share a common goal: identifying a suitable transformation to associate one data set or probability distribution with another. In addition, in many of these applications, optimal transport \cite{Villani2003, V08} plays a pivotal role, as it effectively captures the geometric structure of the data and the underlying spaces.
 
In this paper, we introduce a {\it Wasserstein alignment problem} that unifies a wide variety of alignment problems involving optimal transport. This alignment problem is generally non-convex. Through a novel convex relaxation exploiting statistical independence, we prove a generalized Kantorovich duality that not only characterizes the solution to the original problem, but also sheds light on possible algorithmic development. Our setup has the following components:
\begin{itemize}
\item A source metric space $\X$ equipped with a probability measure $\mu$.
\item A target metric space $\Z$ and a probability measure $\nu$ on $\Z$.
\item A cost function $c : \Z \times \Z \rightarrow \R$ on the target space $\Z$.
\item A family of mappings $\T = \{ T_{\theta} \}_{\theta \in \Theta}$, where each $T_{\theta}$ maps $\X$ to $\Z$ and is parameterized by $\theta$ belonging to a metric space $\Theta$. We denote $T_{\theta}x \equiv T_{\theta}(x)$.
\end{itemize}
We emphasize that the two spaces $\X$ and $\Z$ are allowed to be distinct and have different dimensions. The required regularity conditions are gathered in Assumption \ref{asmp:gendual} below.

For each $\theta \in \Theta$, the pushforward $(T_{\theta})_{\#} \mu := \mu \circ T_{\theta}^{-1}$ is a probability distribution on $\Z$. Using the cost $c$ on $\Z$, we define the optimal transport cost between $(T_{\theta})_{\#} \mu$ and $\nu$ by
\begin{equation} \label{eqn:c.OT.cost}
\wass_c \left(  (T_{\theta})_{\#} \mu, \nu \right) := \inf_{\pi \in \Pi( (T_{\theta})_{\#}\mu, \nu)} \int_{\Z \times \Z} c( y, z) \, d \pi(y, z),
\end{equation}
where $\Pi( (T_{\theta})_{\#}\mu, \nu)$ is the set of couplings of the pair  $((T_{\theta})_{\#}\mu, \nu)$, i.e., the set of joint probability distributions on $\Z \times \Z$ with marginals $(T_{\theta})_{\#}\mu$ and $\nu$. 

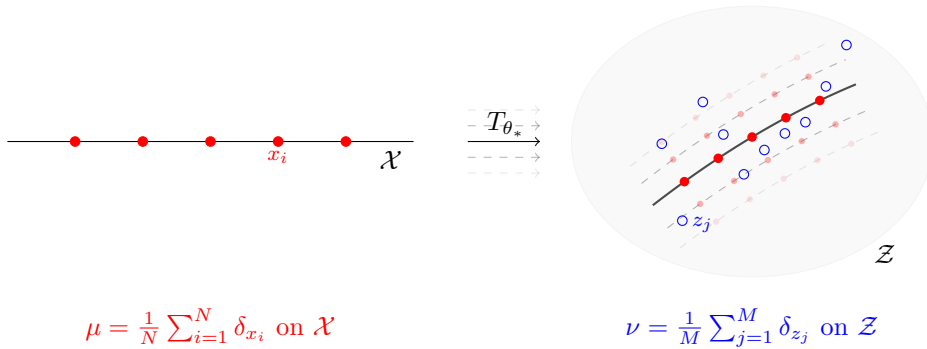
\begin{figure}[t!]
\centering
\begin{tikzpicture}[scale = 0.7]
\pgfmathsetseed{103}

\node[red, below] at (0, -3.5) {$\mu = \sum_{i =1}^N \frac{1}{N} \delta_{x_i}$ on $\X$};
\node[red, below] at (1.5, 0) {\footnotesize $x_i$};
\draw[black] (-4.5, 0) -- (4.5, 0); 
\node[black, below] at (4, 0) {$\mathcal{X}$};
\filldraw[red] (0,0) circle (3pt);
\filldraw[red] (1.5,0) circle (3pt);
\filldraw[red] (3,0) circle (3pt);
\filldraw[red] (-1.5,0) circle (3pt);
\filldraw[red] (-3,0) circle (3pt);

\draw[black, ->, dashed, opacity = 0.1] (5.7, 0.7) -- (7.3, 0.7);
\draw[black, ->, dashed, opacity = 0.3] (5.7, 0.35) -- (7.3, 0.35);
\draw[black, ->] (5.7, 0) -- (7.3, 0);
\draw[black, ->, dashed, opacity = 0.3] (5.7, -0.35) -- (7.3, -0.35);
\draw[black, ->, dashed, opacity = 0.1] (5.7, -0.7) -- (7.3, -0.7);
\node[black, above] at (6.55, -0.1) {$T_{\theta_*}$};

\node[blue, below] at (12, -3.5) {$\nu = \sum_{j = 1}^M \frac{1}{M} \delta_{z_j}$ on $\Z$};
\draw [fill = gray, opacity = 0.05] (12,0) ellipse (4 and 3);
\node[black, right] at (14.5, -2.5) {$\mathcal{Z}$};
\foreach \t in {0,1,...,9} {
        \pgfmathsetmacro\x{10 + 5*\t/11}
        \pgfmathsetmacro\y{-0.9 + 0.3*\t + 1.2*rand}
        \draw[blue] (\x,\y) circle (3pt);
    }
\node[blue, below] at (10.9, -1.46) {\footnotesize $z_j$};

\draw[black, thick, opacity = 0.7, domain=0.6:9.6, variable=\t]
        plot ({9.5 + 5*\t/10}, { -1.4 + 0.3*\t - 0.01*(\t - 5)^2});
\foreach \t in {2,3.5,5,6.5,8} {
        \pgfmathsetmacro\x{9.5 + 5*\t/10}
        \pgfmathsetmacro\y{ -1.4 + 0.3*\t - 0.01*(\t - 5)^2}
        \fill[red] (\x,\y) circle (3pt);
    }
\draw[black, dashed, opacity = 0.3, domain=0.3:9.35, variable=\t]
        plot ({9.4 + 5*\t/10}, { -0.8 + 0.3*\t - 0.01*(\t - 5)^2});
\foreach \t in {1.7,3.2,4.7,6.2,7.7} {
        \pgfmathsetmacro\x{9.4 + 5*\t/10}
        \pgfmathsetmacro\y{ -0.8 + 0.3*\t - 0.01*(\t - 5)^2}
        \fill[red, opacity = 0.3] (\x,\y) circle (2pt);
    }
\draw[black, dashed, opacity = 0.1, domain=0:8.7, variable=\t]
        plot ({9.3 + 5*\t/10}, { -0.3 + 0.3*\t - 0.01*(\t - 5)^2});
\foreach \t in {1.4,2.9,4.4,5.9,7.4} {
        \pgfmathsetmacro\x{9.3 + 5*\t/10}
        \pgfmathsetmacro\y{ -0.3 + 0.3*\t - 0.01*(\t - 5)^2}
        \fill[red, opacity = 0.1] (\x,\y) circle (2pt);
    }       
\draw[black, dashed, opacity = 0.3, domain=0.85:9.9, variable=\t]
        plot ({9.7 + 5*\t/10}, { -2 + 0.3*\t - 0.01*(\t - 5)^2});
\foreach \t in {2.3,3.8,5.3,6.8,8.3} {
        \pgfmathsetmacro\x{9.7 + 5*\t/10}
        \pgfmathsetmacro\y{ -2 + 0.3*\t - 0.01*(\t - 5)^2}
        \fill[red, opacity = 0.3] (\x,\y) circle (2pt);
   }   
\draw[black, dashed, opacity = 0.1, domain=1:10, variable=\t]
        plot ({9.9 + 5*\t/10}, { -2.5 + 0.3*\t - 0.01*(\t - 5)^2});
\foreach \t in {2.6,4.1,5.6,7.1,8.6} {
        \pgfmathsetmacro\x{9.9 + 5*\t/10}
        \pgfmathsetmacro\y{ -2.5 + 0.3*\t - 0.01*(\t - 5)^2}
        \fill[red, opacity = 0.1] (\x,\y) circle (2pt);
    }  
\end{tikzpicture}
\caption{Illustration of the Wasserstein alignment problem \cref{eq:wgencost}, where both $\mu$ and $\nu$ are discrete measures. Given a family $\{T_{\theta}\}_{\theta \in \Theta}$ of continuous mappings from $\X$ to $\Z$, we wish to find $\theta_* \in \Theta$ such that the pushforward $(T_{\theta_*})_{\#} \mu$ is closest to $\nu$ with respect to the optimal transport cost $\wass_c$. {\bf Left}: $\mu$ is the point cloud with masses at $x_i$ represented by solid red circles. {\bf Right}: $\nu$ is the point cloud with masses at $z_j$ represented by hollow blue circles. The solid curve and the solid red circles represent the images of $\X$ and $x_i$ under the optimal map $T_{\theta_*}$. The other curves and circles are images under suboptimal $T_{\theta}$.} \label{fig:alignment.illustration}
\end{figure}

The {\it Wasserstein alignment problem} considered in this paper is the optimization problem
\begin{equation} \label{eq:wgencost}
\wass_c^{\T} (\mu, \nu) := \inf_{\theta \in \Theta} \wass_c\left(  (T_{\theta})_{\#} \mu, \nu \right),
\end{equation}
where the superscript $\T$ refers to the family $\{T_{\theta} : \X \rightarrow \Z \}$ of transformations. This problem is asymmetric in $\mu$ and $\nu$ since each $T_{\theta}$ pushforwards the source distribution $\mu$ on $\X$ to a distribution on the target space $\Z$. More generally, we may consider additionally a \textit{penalization function} $R:\Theta \rightarrow \rr$, which leads to the {\it penalized Wasserstein alignment problem}
\begin{equation} \label{eq:wgencostpenal}
\wass_{c}^{\T, R} (\mu, \nu) := \inf_{\theta \in \Theta} \left\{ \wass_c\left(  (T_{\theta})_{\#} \mu, \nu \right) + R(\theta) \right\}.
\end{equation}
The general idea is that we would like to find a transformation $T_{\theta_*}$ that aligns most closely, in an optimal transport sense, $\mu$ to $\nu$ on $\Z$. See \cref{fig:alignment.illustration} for an illustration when $R \equiv 0$. In practice, the penalization term $R(\theta)$ can be thought of as penalizing the {\it complexity} of the transformation $T_\theta$. 

We list a few examples where special cases of this, or similar, problems have been studied.\footnote{See \cref{sec:further.examples} for further examples.} In particular, in \cref{eg:GW} below we show that the {\it inner product Gromov--Wasserstein distance} can be written in the form of~\cref{eq:wgencostpenal} for a suitable $R(\theta)$.

\begin{example}[Procrustes problem] \label{eg:Procrustes}
In the Procrustes problem, one seeks an optimal transformation---such as translation, scaling, rotation, or reflection---that best aligns one data matrix, $\mathbf{X} \in \R^{N \times n}$, with another, $\mathbf{Y} \in \R^{N \times d}$~\cite{Gower-Dijksterhuis-2004}. Specifically, the objective is to find a transformation matrix $T \in \R^{n \times d}$ that solves $\min_{T} \|\mathbf{X} T - \mathbf{Y}\|^2$, where $\|\cdot\|$ denotes the Frobenius norm. The matrix $T$ can be restricted to different subclasses of linear transformations.  In essence, Procrustes analysis learns a linear transformation that maps one set of {\it matched} points -- represented by the rows of $\mathbf{X}$ and $\mathbf{Y}$ -- onto the other. A particularly well-studied case is the {\it orthogonal Procrustes problem}, where $T$ is restricted to be an orthogonal matrix (i.e., $n=d$ and $T^\top T = I_n$). This variant has a closed form solution~\cite{Schonemann-1966} and has gained renewed interest in machine learning~\cite{Xing-2015}. The Procrustes problem has widespread applications, from multivariate statistics to machine learning, including shape analysis in 2D~\cite{Dryden-Mardia-2016, Goodall-1991} and learning a linear mapping between word embeddings in different languages using a bilingual lexicon~\cite{Mikolov-2013}. Our problem~\cref{eq:wgencost} extends the Procrustes problem by allowing for an arbitrary family $\mathcal{T}$  of transformations as opposed to just linear transformations. More importantly, we remove the assumption that the correspondences between the two sets are known a priori. Instead, we leverage the Wasserstein distance in~\cref{eq:wgencost} to infer these correspondences.
\end{example}

\begin{example}[Point-cloud registration or scan matching] \label{eg:point-cloud}
Point-cloud registration is an important problem for aligning 2D/3D shapes with applications in shape analysis, computer vision, robotics, and medical imaging~\cite{besl1992method, Dryden-Mardia-2016, Srivastava-2016}. 
 The problem is to find an optimal spatial transformation (scaling, rotation and translation) that minimizes the discrepancy between two point clouds. The Iterative Closest Point or ICP algorithm is a widely used method to solve this problem. The algorithm iteratively refines the alignment by: (1) finding correspondences by pairing each point in one data set with its ``nearest'' point (either in a greedy or an optimal matching sense) in the other, (2) estimating a transformation that minimizes the sum of squared distances between the matched points, and (3) applying the transformation and repeating until convergence. Despite its efficiency, ICP suffers from sensitivity to initialization, often converging to local minima if the initial alignment is poor. Additionally, it is vulnerable to noise and outliers. Many variants of the ICP algorithm have been developed over the years to improve its accuracy, robustness, and convergence properties~\cite{pomerleau2013, rusinkiewicz2001}. Our objective in~\cref{eq:wgencostpenal} indeed tackles a problem similar to that of the ICP algorithm. However, our proposed convex Kantorovich-type dual formulation (see \cref{thm:duality}) overcomes many of the limitations inherent to ICP, offering a theoretically grounded approach.
\end{example}

\begin{example}[Gromov--Wasserstein alignment] \label{eg:GW}
The Gromov--Wasserstein (GW) distance \cite{Memoli-2011} generalizes the standard optimal transport framework to compare two metric measure spaces. Unlike traditional optimal transport, which relies on a given cost function to transport from one space to the other, GW instead compares pairwise distances or dissimilarities within each space under suitable couplings. This intrinsic approach makes GW particularly well suited for aligning distributions that lack a common ambient space. The GW framework has been applied across a wide range of domains involving heterogeneous data, including single-cell genomics~\cite{Blumberg-2020, CWZ25, Demetci2022}, alignment of language models~\cite{Alvarez2018}, shape and graph matching~\cite{Koehl-2023,Memoli-2009,Xu-2019}, heterogeneous domain adaptation~\cite{Yan-2018}, and generative modeling~\cite{Bunne-2019}.

A variant of GW, called the {\it inner product} GW or IGW, can be reduced to our setup. The IGW distance $\mathrm{IGW}(\mu, \nu)$ between probability measures $\mu$ on $\R^n$ and $\nu$ on $\R^d$ is defined as
\begin{equation} \label{eqn:IGW}
\mathrm{IGW}^2(\mu, \nu) :=\inf_{\pi \in \Pi(\mu, \nu)} \int_{(\R^n \times \R^d)^2} \abs{x \cdot x' - z \cdot z'}^2 d\pi(x,z) d\pi(x', z'),
\end{equation}
where $x \cdot x'$ and $z \cdot z'$ are inner products on $\R^n$ and $\R^d$ respectively. It is shown in~\cite[Lemma 2.1]{Zhang-2024} that $\mathrm{IGW}^2(\mu, \nu)$ is a sum of two quantities $F_1 + F_2$ where $F_1$ depends only on the marginals $\mu$ and $\nu$, and $F_2$ captures the actual optimization over couplings. Moreover, it follows using a dual representation of \cref{eqn:IGW} that (also see \cref{eqn:wass.c.change.coupling}),
\begin{equation}\label{IGW}
F_2= \inf_{A\in \rr^{n \times d}} \left\{ \wass_c((A^\top)_\#\mu, \nu) + 8 \norm{A}^2 \right\},
\end{equation}
where $(A^\top)_\#\mu$ is the pushforward of $\mu$ under the transpose mapping $x \mapsto A^{\top} x$, and the infimum is over all $n \times d$ matrices $A$. Here, $c(x,z)=-8x \cdot  z$ is the inner product cost function on $\R^d$ and $\norm{A}$ refers to the Frobenius norm of $A$. This is exactly of the form \cref{eq:wgencostpenal}. An analogous variational representation for an entropic Gromov--Wasserstein distance was given in  \cite[Theorem 1]{ZGMS24}. These results, and ours, demonstrate the power of duality in elucidating the hidden structures of optimal transport problems beyond the classical setting.
\end{example}

\begin{example}[Wasserstein projection] \label{eg:Wasserstein.projection}
Returning to \cref{eq:wgencost}, we see that $\wass_c^{\T} (\mu, \nu)$ can be regarded as a projection problem in optimal transport. Let $\T_{\#}\mu := \{ (T_{\theta})_{\#} \mu: \theta \in \Theta\}$ be the subset of the space $\Prob(\Z)$ of all probability measures on $\Z$ comprising of the pushforwards of $\mu$ by functions in $\T$. Then the $\argmin_{\nu' \in \T_{\#} \mu} \wass_c(\nu', \nu)$, if it exists, can be thought of as a projection of $\nu$ onto $\T_{\#}\mu$ under the $\wass_c$-loss. Our work is thus related to optimal transport between different spaces \cite{CMP17, MP20} as well as Wasserstein projections in various senses \cite{ACJ20, CG03}.

In fact, projections on arbitrary subsets of the Wasserstein space may be regarded as an example of our setup. For simplicity, take $\X = \Z = \R^n$ and $c(y, z) = |y - z|^2$ to be the quadratic cost. Let $\{\mu_\theta : \theta \in \Theta\}$ be any subset of the $2$-Wasserstein space of probability measures on $\rr^n$, parametrized by $\Theta$. Consider a probability distribution $\mu$ on $\rr^n$ that is absolutely continuous (say, standard normal). Thus, by Brenier's theorem, for every $\mu_\theta$, there exists a map $T_\theta$ such that $\left(T_\theta\right)_{\#}\mu= \mu_\theta$. Letting $\T = \{T_{\theta} : \theta\in \Theta\}$, we see that for any other probability measure $\nu$, we have $\wass_c^\T(\mu, \nu)= \inf_{\theta \in \Theta} \wass_2\left(\mu_\theta, \nu \right)$. 
\end{example}

The (penalized) Wasserstein alignment problem \cref{eq:wgencostpenal} is generally nonconvex in the sense that the objective $\wass_c \left(  (T_{\theta})_{\#} \mu, \nu \right) + R(\theta)$ is generally not convex in $\theta \in \Theta$ (when $\Theta \subset \R^k$ for some $k$). An explicit example is given in \cref{sec:appendix.1}. Although the existence of an optimal $\theta \in \Theta$ for \cref{eq:wgencost} can be proved under mild conditions (see \cref{prop:existence}), 
one needs additional conditions to show uniqueness (see \cref{prop:uniqueness}) and stability of the solution when the source and target measures are perturbed. The {\it value} $\wass_c^{\T, R}(\mu, \nu)$ is nevertheless stable under perturbation of the source and target measures (see \cref{prop:existence}). 

In optimization in general and optimal transport in particular, a general strategy for dealing with new problems is to formulate suitable relaxations under which the problem becomes convex. Nevertheless, its implementation differs case by case depending on the structure of the problem in hand.  One of our main contributions in this paper is to formulate, under mild conditions, such a generalized Kantorovich duality for the (penalized) Wasserstein alignment problem \cref{eq:wgencostpenal}. In particular, it leads naturally to a sufficient and necessary condition for optimality. We expect that this theoretical insight will be useful in the study of specific alignment problems, where additional structural conditions can be exploited. 

To state this duality result, let us first introduce some notation. For a topological space $E$, let $C(E)$ be the set of all real-valued continuous functions on $E$.

\begin{definition}[Function class for the dual problem]\label{defn:whatisCEF}
Given $\mu \in \Prob(\X)$, let $\CEF_{\mu}$ denote the set of $\xi  \in C(\X \times \Theta)$  such that the integral $\int_{\X} \xi(x, \theta) d\mu(x)$ exists and is a constant function of $\theta \in \Theta$. This constant value will be denoted by $\int \xi(x, \cdot) d\mu(x)$.
\end{definition}

\begin{assumption} \label{asmp:gendual}
We assume the following conditions:
\begin{itemize}[itemindent=1.5cm] 
\item[(i)] $(\X, d_{\X})$ and $(\Z, d_{\Z})$ are compact metric spaces.
\item[(ii)] $(\Theta, d_{\Theta})$ is a compact metric space.
\item[(iii)] $T: \Theta \times \X \mapsto \Z$ defined by $T(\theta, x) = T_{\theta}x$ is continuous. 
\item[(iv)] $c: \Z \times \Z \rightarrow \R$ is jointly continuous.
\item[(v)] $R: \Theta \rightarrow \R$ is continuous.
\end{itemize}
\end{assumption}

\begin{remark}[Comments on the assumptions]
Part (i) of the assumptions may be equivalently stated as $\mu$ and $\nu$ are compactly supported on their ambient spaces. This compactness assumption may be relaxed at the cost of introducing additional integrability conditions on $\mu$ and $\nu$. For the Euclidean setting (where $\X = \R^n$ and $\Z = \R^d$) described in \cref{sec:Euclidean}, we only assume, at least initially, that $\mu$ and $\nu$ have finite second moments that are suitably normalized. Compactness of $\Theta$ is natural in many geometric settings. For example, in the Euclidean setting $\Theta$ corresponds to a set orthogonal transformations. Even if $\Theta$ is not compact, as long as $\wass( (T_{\theta})_{\#}\mu, \nu)$ diverges when $\theta$ tends to infinity, we may replace $\Theta$ by a compact subset of it. Finally, note that in (iv) we allow $c(y, z)$ to be asymmetric in $(y, z)$. 
\end{remark}

We now state our main result which will be proved in \cref{sec:duality}.

\begin{theorem}[Generalized Kantorovich duality] \label{thm:duality}
Under \cref{asmp:gendual}, 
    \begin{equation}\label{eq:cxdual}
            \wass_c^{\T, R}(\mu, \nu)= \sup\left\{ \int_{\X} \xi(x, \cdot) d\mu(x) + \int_{\Z} \psi(z) d\nu(z) \right\},
        \end{equation}
        where the supremum is over all functions $(\xi,\psi)\in \CEF_{\mu} \times C(\Z)$ satisfying the constraint
        \begin{equation}\label{eq:dualconstraint}
            \qquad \;\; \xi(x, \theta) + \psi(z) \le c\left(T_{\theta} x, z \right) + R(\theta), \quad \mbox{for all} \;\; (x, \theta, z) \in \X \times \Theta \times \Z.
        \end{equation}
\end{theorem}
Note that the above problem is convex in its argument $(\xi, \psi)$. Hence we can implement the non-convex Wasserstein alignment problem by running a {\it linear program}.  In \cref{sec:implementation} we provide a linear programming formulation in a  discretized setting and  illustrate its use with several examples from shape analysis. We also discuss possible ways of making this approach more scalable, but a full algorithmic development is beyond the scope of this paper.

\cref{thm:duality} is a consequence of a {\it convex relaxation} of \cref{eq:wgencostpenal} that allows $\theta \in \Theta$ to be {\it randomized}. Define
\begin{equation}\label{eq:convexify}
 \overline{\wass}_c^{\T, R}(\mu, \nu) := \inf_{\gamma \in \Upsilon(\mu, \nu)} \left\{ \int_{\X \times \Theta \times \Z} \left(c\left(T_{\theta}x, z\right) + R(\theta) \right) d\gamma(x, \theta, z) \right\}.
\end{equation}
Here, $\Upsilon(\mu, \nu)$ is the set of Borel probability measures on $\X \times \Theta \times \Z$ such that if $(X, \eta, Z) \sim \gamma \in \Upsilon(\mu, \nu)$ then (i) $X \sim \mu$ and $Z \sim \nu$, and (ii) $X$ and $\eta$ are independent. It is shown in \cref{lem:Krelax} that $\overline{\wass}_c^{\T, R}(\mu, \nu) = \wass_c^{\T, R}(\mu, \nu)$ and it is via \cref{eq:convexify} that we obtain the dual \cref{eq:cxdual}.

\begin{remark}[Discussion]
The reader might wonder what the intuition is behind demanding $X$ and $\eta$ to be independent in \cref{eq:convexify}. In the Wasserstein alignment problem, we need to pick a {\it single} $\theta$ so as to make $(T_{\theta})_{\#} \mu$ as close as possible to $\nu$. If the randomized parameter $\eta$ and $X \sim \mu$ were allowed to be dependent, one could use the {\it realized value} of $X$ to pick a favorable value of $\theta$. This strictly violates the set-up of our alignment problem. The independence assumption is heavily used in our proof of \cref{thm:duality}.
\end{remark}

As a corollary, we provide a necessary and sufficient condition for $\theta_* \in \Theta$ to be optimal for $\wass_c^{\T, R}(\mu, \nu)$. Before presenting this result, we introduce some notation. The following definition is a standard concept in optimal transport theory (see, e.g.,~~\cite[Definition 1.10]{santam2015ot}). 

\begin{definition}[$c/\overline{c}$- transform ]\label{defn:cconcave}
Given a function $\psi: \Z \rightarrow \overline{\R} := [-\infty, \infty]$, we define its $c$-transform $\psi^c$ by $\psi^c(z) := \inf_{y \in \Z}\left\{ c(y,z) - \psi(y) \right\}$, $z \in \Z$, and its $\overline{c}$-transform by $\psi^{\overline{c}}(y) := \inf_{z \in \Z}\left\{ c(y,z) - \psi(z) \right\}$, $y \in \Z$. We say that $\psi : \Z \rightarrow \R \cup \{-\infty\}$ is $c$-concave (resp.~$\overline{c}$-concave) if $\psi \not\equiv -\infty$ and $\psi = \varphi^{\overline{c}}$ (resp.~$\psi = \varphi^c$) for some $\varphi$. We let $\cconv(\Z)$ (resp.~$\cbarconv(\Z)$) be the set of $c$-concave (resp.~$\overline{c}$-concave) functions.
\end{definition}

As we allow $c(y, z)$ to be asymmetric in $y$ and $ z$, the $c$ and $\overline{c}$-transforms are generally different. Clearly, if $c$ is symmetric (e.g.,~the quadratic cost $c(y, z) = \|y - z\|^2$), then the two transforms coincide. In this case, we call both $\psi^{c}$ and $\psi^{\overline{c}}$ the $c$-transform. It follows that $\psi^{\bar{c}}(y) + \psi(z) \leq c(y, z)$ for all $y, z \in \Z$. For $\psi \in C(\Z)$, define $I_{\psi} : \Theta \rightarrow \R$ by
\begin{equation}\label{eqn:whatisIpsi}
I_\psi(\theta):= \int_{\X} \psi^{\overline{c}}(T_{\theta} x) d\mu(x) + R(\theta).
\end{equation}
Combining \cref{thm:duality} with $\overline{c}$-concavity, we obtain the following result proved in \cref{sec:duality}.

\begin{corollary}[Optimality criterion] \label{cor:optimality}
Suppose Assumption \ref{asmp:gendual} holds. Then $\theta\in \Theta$ is optimal for $\wass_c^{\T, R}(\mu, \nu)$ if and only if $\theta \in \argmin_\Theta I_{\psi}(\cdot)$ for some pair $(\psi^{\overline{c}}, \psi)$ of Kantorovich potentials for $\wass_c\left( (T_{\theta})_{\#}\mu, \nu \right)$. In fact, for any $\theta \in \Theta$ and any pair $(\psi^{\overline{c}}, \psi)$ of Kantorovich potentials for $\wass_c\left( (T_{\theta})_{\#}\mu, \nu \right)$, we have
\begin{equation} \label{eqn:gap}
0 \leq \left( \wass_c\left( (T_{\theta})_{\#} \mu, \nu \right) + R(\theta) \right) - \wass_c^{\T, R}(\mu, \nu) \leq I_{\psi}(\theta) - \min_{\Theta} I_{\psi}(\cdot).
\end{equation}
\end{corollary}

A particular example that we will develop in further detail is the Euclidean setting, where $\X = \R^n$, $\Z = \R^d$ with $n \geq d$, $c(y, z) = \|y - z\|^2$ is the quadratic cost on $\R^d$, and each $T_{\theta}$ is a linear transformation of the form $T_\theta(x) = A^\top x$ for a matrix $A \in \R^{n \times d}$ with orthonormal columns. This setting is closely related to the {\it sliced Wasserstein distance} \cite{BRPP15} which is also based on linear projections of measures. Here, we can analogously define another Wasserstein alignment problem, in the opposite direction, using the transposes $A^{\top}: \R^d \rightarrow \R^n$ and the quadratic cost on $\R^n$. Remarkably, we show in \cref{prop:Euclidean.equivalence} that the two   alignment problems (from $\R^n$ to $\R^d$ and from $\R^d$ to $\R^n$) are, in fact, equivalent when $\mu$ and $\nu$ are normalized to have zero means and identity covariances. Exploiting the Lie group structure of orthogonal matrices, we show in \cref{thm:downfoc} that the cross-correlation of the optimal projection must satisfy a striking symmetry condition. This Euclidean case will be described in \cref{sec:Euclidean} and further developed in \cref{sec:Euclidean.further}. 

\subsection{Open questions} 
We conclude this introduction by highlighting some open problems. 

Our first problem concerns implementation of the dual problem \cref{eq:cxdual}. When everything is discrete we have provided a linear programming formulation of this problem in \cref{sec:implementation}. This, unfortunately, requires $\Theta$ to be discrete as well, which is not ideal. Can we solve this convex problem efficiently for an arbitrary $\Theta$ between two empirical distributions $\mu$ and $\nu$?

Speaking of empirical distributions, \cref{lem:existence}, part (ii), shows that a sample estimate of $\wass_c^\T(\mu, \nu)$ is consistent under mild assumptions. Recently, the convergence rate of estimation of $\wass_c^\T(\mu, \nu)$ from samples from $\mu$ and $\nu$ was established in \cite{KW25}. How about the sample complexity of the optimal transformation $T_{\theta_*}$?

Computations in the usual optimal transport  are usually tackled by {\it entropic regularization}. Here too one can formulate a natural entropic regularized problem. Suppose $\X, \Z, \Theta$ are all measurable subsets of Euclidean spaces. For any Borel probability distribution $\gamma$ on $\X \times \Theta \times \Z$, we say it is absolutely continuous if it is absolutely continuous with respect to the product (restricted) Lebesgue measure. If $\gamma$ is absolutely continuous (and identified with its density), define its entropy as $\Ent(\gamma) := \int \gamma \log \gamma \, dxd\theta dz$. Take $\Ent(\gamma)=\infty$ if $\gamma$ is not absolutely continuous. Consider the convex relaxation \cref{eq:convexify}. If one modifies it to, for $\epsilon >0$, 
\begin{equation*}
\overline{\wass}_{c, \epsilon}^{\T, R}(\mu, \nu) := \inf_{\gamma \in \Upsilon(\mu, \nu)} \left\{ \int_{\X \times \Theta \times \Z} \left( c\left(T_{\theta}x, z\right) + R(\theta) \right) d\gamma(x, \theta, z) + \epsilon  \Ent(\gamma) \right\},
\end{equation*}
then the problem is strictly convex in $\gamma$ and therefore, admits, a unique minimum. It is not difficult to see that $\lim_{\epsilon \rightarrow 0^+}\overline{\wass}_{c, \epsilon}^{\T, R}(\mu, \nu)=\wass_c^{\T, R} (\mu, \nu)$. The question is whether one may use the Sinkhorn algorithm to solve this new variant? What is its sample complexity?

Many of the natural computational methods for similar problems hinge on the uniqueness of the minimizer. Although in  \cref{prop:uniqueness} we prove uniqueness under somewhat stringent conditions, it remains a delicate theoretical open problem to formulate natural sufficient conditions to guarantee uniqueness.

\section{Wasserstein alignment} \label{sec:set-up}
This section has two purposes. First, we establish in \cref{sec:general.formulation} some basic existence and stability results for our penalized Wasserstein alignment problem introduced in \cref{eq:wgencostpenal}. Second, we introduce the Euclidean setting in \cref{sec:Euclidean} which will be further developed in \cref{sec:Euclidean.further}.

\subsection{Basic properties} \label{sec:general.formulation}
We will work under \cref{asmp:gendual}. Given a topological space $E$, let $\Prob(E)$ denote the set of all Borel probability measures on $E$. In the following, we use $w, x$ to denote elements of $\X$, and $y, z$ for elements of $\Z$. We equip $\Prob(\X)$ and $\Prob(\Z)$ with the topology of weak convergence.  Since $\X$ and $\Z$ are compact, by Prokhorov's theorem, $\Prob(\X)$ and $\Prob(\Z)$ are compact as well. Another useful consequence of \cref{asmp:gendual} is that $T$, $c$ and $R$ are {\it uniformly continuous} on their respective domains. For each $\theta \in \Theta$, recall the quantity $\wass_c\left( (T_{\theta})_{\#} \mu, \nu \right)$ defined by \cref{eqn:c.OT.cost}. 

It is standard (see e.g.~\cite[Theorem 4.1]{V08}) to show that $\wass_c\left( (T_{\theta})_{\#} \mu, \nu \right)$ is finite and is attained by an optimal coupling $\pi \in \Pi ( (T_{\theta})_{\#}\mu, \nu)$. Minimizing over $\theta \in \Theta$ yields 
$\wass_c^{\T}(\mu, \nu)$, as defined in \cref{eq:wgencost}. When a penalization term 
$R$ is included, this results in $\wass_c^{\T, R}(\mu, \nu)$, as defined in \cref{eq:wgencostpenal}. We say that $\theta_* \in \Theta$ is {\it optimal} for $\wass_c^{\T, R}(\mu, \nu)$ if $\wass_c\left( (T_{\theta_*})_{\#} \mu, \nu \right) + R(\theta_*) = \wass_c^{\T, R}(\mu, \nu)$. When $R \equiv 0$, we say that $\theta_*$ is optimal for $\wass_c^{\T}(\mu, \nu)$.

An important special case is where the cost function $c(\cdot,\cdot)$ is a power of the metric $d_{\Z}$ on $\Z$, that is, $c(y, z) = d^p_{\Z}(y, z)$ for some $p \geq 1$. Then $\wass_c = \wass_p^p$, where
\[
\wass_p(\nu_0, \nu_1) := \left( \inf_{\pi \in \Pi(\nu_0, \nu_1)} \int d_{\Z}^p(y, z) d \pi(y, z) \right)^{1/p}
\]
is the {\it $p$-Wasserstein distance} between $\nu_0, \nu_1 \in \Prob(\Z)$. Since $\Z$ is assumed to be compact, $\Prob_p(\Z)$, the set of Borel probability measures on $\Z$ with finite $p$-th moment, reduces to $\Prob(\Z)$. Also, convergence in $\wass_p$ is equivalent to weak convergence. In this case, we define
\begin{equation} \label{eqn:downward.Wasserstein}
\wass_p^{\T}(\mu, \nu) := \inf_{\theta \in \Theta} \wass_p\left( (T_{\theta})_{\#} \mu, \nu \right).
\end{equation}

We begin with the following lemma which gives some basic properties of the optimal transport cost $\wass_c \left(  (T_{\theta})_{\#} \mu, \nu \right)$ defined in~\cref{eqn:c.OT.cost}. 

\begin{lemma} \label{lem:existence} { \ }
Under \cref{asmp:gendual} we have the following results:
\begin{itemize}[itemindent = 1.5cm]
\item[(i)] For any $\theta \in \Theta$ we have
\begin{equation} \label{eqn:wass.c.change.coupling}
\wass_c \left(  (T_{\theta})_{\#} \mu, \nu \right) = \inf_{\gamma \in \Pi( \mu, \nu)} \int_{\X \times \Z} c\left( T_{\theta} x, z \right) d \gamma (x, z);
\end{equation}
that is, we may directly optimize over couplings of $\mu$ and $\nu$.
\item[(ii)] $\wass_c \left( (T_{\theta})_{\#}\mu, \nu \right)$ is jointly continuous in $(\mu, \theta, \nu) \in \Prob(\X) \times \Theta \times \Prob(\Z)$.
\end{itemize} 
\end{lemma}
\begin{proof}
(i) Clearly, if $\gamma$ is coupling of $(\mu, \nu)$, then $\pi = (T_{\theta}, \id)_{\#} \gamma$ is a coupling of $\left((T_{\theta})_{\#}\mu, \nu \right)$ and
\[
\int_{\X \times \Z} c\left( T_{\theta} x, z \right) d \gamma(x, z) = \int_{\Z \times \Z} c(y, z) d \pi(y, z) \geq \wass_c \left(  (T_{\theta})_{\#} \mu, \nu \right).
\]
Taking infimum over $\gamma$ shows that the inequality $\leq$ holds in \cref{eqn:wass.c.change.coupling}. Next, using the disintegration theorem, define $\mu (dx | y)$ to be the conditional distribution of $X$ given $T_{\theta} X = y$ if unconditionally $X \sim \mu$. For each $\pi \in \Pi\left( (T_{\theta})_{\#}\mu, \nu\right)$, define $\gamma$ by the $(x, z)$-marginal of the probability measure $\mu(dx | y) \pi(dy, dz)$. It is easy to see that $\gamma \in \Pi(\mu, \nu)$ and
\begin{equation*}
\begin{split}
\int_{\X \times \Z} c(T_{\theta}x, z) d \gamma(x, z) &= \int_{\X \times \Z \times \Z} c(T_{\theta}x, z)  \mu(dx |y) \pi(dy, dz) \\
&= \int_{\Z \times \Z} c(y, z) \pi(dy, dz),
\end{split}
\end{equation*}
since $T_{\theta}x = y$ almost surely under $\mu(dx | y) \pi(dy, dz)$.  Thus the reverse inequality holds as well.

(ii) Let $\mu_k \rightarrow \mu_{\infty} \in \Prob(\X)$, $\theta_k \rightarrow \theta_{\infty} \in \Theta$ and $\nu_k \rightarrow \nu_{\infty} \in \Prob(\Z)$. We will show that for any subsequence $k'$ there exists a further subsequence $k''$ along which $\wass_c\left( (T_{\theta_{k''}})_{\#} \mu_{k''}, \nu_{k''}\right) \rightarrow \wass_c\left( (T_{\theta_{\infty}})_{\#} \mu_{\infty}, \nu_{\infty}\right)$. Clearly, this implies that the original sequence converges to the same limit. Thus consider a subsequence $k'$. For each $k'$, let $\gamma_{k' }\in \Pi(\mu_{k'}, \nu_{k'})$ be optimal for $\wass_c\left( (T_{\theta_{k'}})_{\#} \mu_{k'}, \nu_{k'}\right)$. We claim that $c(T_{\theta_k} x, z) \rightarrow c(T_{\theta_{\infty}} x, z)$ uniformly in $(x, z) \in \X \times \Z$. Since $T$ is uniformly continuous on $\X \times \Z$, there exists $\omega_T: [0, \infty) \rightarrow [0, \infty)$ with $\lim_{t \downarrow 0} \omega_T(t) = \omega_T(0) = 0$, such that
\[
d_{\Z} \left( T_{\theta}x, T_{\theta'} x'\right) \leq \omega_T\left( d_{\Theta}(\theta, \theta') + d_{\X}(x, x')\right), \quad (\theta, x), (\theta', x') \in \Theta \times \X.
\]
We call $\omega_T$ a {\it modulus of continuity} for $T$. Similarly, there exists a modulus of continuity $\omega_c$ for $c$, such that
\[
|c(y, z) - c(y', z')| \leq \omega_c \left( d_{\Z}(y, y') + d_{\Z}(z, z') \right), \quad (y, z), (y', z') \in \Z \times \Z.
\]
It follows that
\[
\sup_{x \in \X, z \in \Z} \left| c\left( T_{\theta_k}x, z \right) - c\left( T_{\theta_{\infty}}x, z \right) \right| \leq \omega_c \left( \omega_T \left( d_{\Theta}(\theta_k, \theta) \right) \right) \rightarrow 0, \quad k \rightarrow \infty.
\]

By \cite[Theorem 5.20]{V08}, we obtain a further subsequence $k''$ along which $\gamma_{k''}$ converges weakly to some $\gamma_{\infty} \in \Pi(\mu, \nu)$ which is optimal for $\wass_c \left( (T_{\theta_{\infty}})_{\#} \mu_{\infty}, \nu_{\infty} \right)$. Since $\X$ and $\Z$ are Polish, the Skorokhod representation theorem implies that on some probability space there exist random elements $(X_{k''}, Z_{k''})$ such that $(X_{k''}, Z_{k''}) \sim \gamma_{k''}$ and $(X_{k''}, Z_{k''}) \rightarrow (X_{\infty}, Z_{\infty})$ almost surely with $(X_{\infty}, Z_{\infty}) \sim \gamma$. Since $c$ is bounded and $c(T_{\theta_{k''}} X_{k''}, Z_{k''}) \rightarrow c(T_{\theta_{\infty}} X_{\infty}, Z_{\infty})$ almost surely, the bounded convergence theorem gives
\begin{equation*}
\begin{split}    
\lim_{{k''} \rightarrow \infty} \wass_c \left( (T_{\theta_{k''}})_{\#}\mu_{k''}, \nu_{k''} \right) &= \lim_{{k''} \rightarrow \infty} \E_{(X_{k''}, Z_{k''}) \sim \gamma_{k''}} \left[ c(T_{\theta_{k''}} X_{k''}, Z_{k''})  \right] \\
  &= \E_{(X_{\infty}, Z_{\infty}) \sim \gamma_{\infty}} \left[ c(T_{\theta_{\infty}} X_{\infty}, Z_{\infty})  \right] \\
  &= \wass_c \left( (T_{\theta_{\infty}})_{\#} \mu_{\infty}, \nu_{\infty} \right).
\end{split}
\end{equation*}
Hence $\wass_c \left( (T_{\theta})_{\#}\mu, \nu \right)$ is jointly continuous in $(\mu, \theta, \nu)$.
\end{proof}

\begin{proposition}[Existence and stability] \label{prop:existence} The following holds under \cref{asmp:gendual}.
\begin{itemize}[itemindent=1.5cm]
    \item[(i)] There exists $\theta_* \in \Theta$ which is optimal for $\wass_c^{\T, R}(\mu, \nu)$.
    \item[(ii)] $\wass_c^{\T, R}(\mu, \nu)$ is continuous in $(\mu, \nu) \in \Prob(\X) \times \Prob(\Z)$.
    \item[(iii)] If $\mu_k \rightarrow \mu_{\infty} \in \Prob(\X)$ and $\nu_k \rightarrow \nu_{\infty} \in \Prob(\Z)$, and if $\theta_k$ is optimal for $\wass_c^{\T, R}(\mu_k, \nu_k)$, then any limit point $\theta_{\infty}$ of the sequence $(\theta_k)_{k \ge 1}$ in $\Theta$ is optimal for $\wass_c^{\T, R}(\mu_{\infty}, \nu_{\infty})$. 
\end{itemize}
\end{proposition}
\begin{proof}
(i) For $\mu \in \Prob(\X)$ and $\nu \in \Prob(\Z)$ fixed, the objective $\wass_c \left(  (T_{\theta})_{\#} \mu, \nu \right) + R(\theta)$ is continuous in $\theta$ from \cref{lem:existence}(ii) and the continuity of $R$. Since $\Theta$ is compact, the existence of an optimal $\theta_*$ follows from the Weierstrass theorem.

(ii) Suppose $\mu_k \rightarrow \mu_{\infty} \in \Prob(\X)$ and $\nu_k \rightarrow \nu_{\infty} \in \Prob(\Z)$.\footnote{Implicitly, we will be taking a subsequence as in the proof of \cref{lem:existence}(ii).} For each $k$, let $\theta_{k} \in \Theta$ be optimal for $\wass_c^{\downarrow, R}(\mu_{k}, \nu_{k})$. 
Since $\Theta$ is compact, by passing along a subsequence, we may assume $\theta_{k} \rightarrow \theta_{\infty} $ for some $\theta_{\infty} \in \Theta$. From \cref{lem:existence}(ii), we have
\begin{equation} \label{eqn:lsc} 
\begin{split}
\wass_c^{\downarrow, R}(\mu_{\infty}, \nu_{\infty})  &\leq \wass_c \left( (T_{\theta_{\infty}})_{\#} \mu_{\infty}, \nu_{\infty} \right) + R(\theta_{\infty}) 
= \lim_{k \rightarrow \infty} \wass_c^{\downarrow, R}(\mu_{k}, \nu_{k}).   
\end{split}
\end{equation}
To see that equality holds, let $\theta_* \in \Theta$ be optimal for $\wass_c^{\downarrow, R}(\mu_{\infty}, \nu_{\infty})$, and let $(\tilde{\theta}_{k}) \subset \Theta$ be any sequence that converges to $\theta_*$. Using \cref{lem:existence}(ii) again, we have
\begin{equation*}
\begin{split}
\wass_c^{\downarrow, R}(\mu_{\infty}, \nu_{\infty}) &= \wass_c\left( (T_{\theta_*})_{\#} \mu_{\infty}, \nu_{\infty}\right) + R(\theta_*) \\
&= \lim_{k \rightarrow \infty}\left( \wass_c \left(  (T_{\tilde{\theta}_{k}})_{\#} \mu_{k}, \nu_{k} \right) + R(\tilde{\theta}_k) \right) \\
&\geq \lim_{k \rightarrow \infty} \wass_c^{\downarrow, R}(\mu_{k}, \nu_{k}). 
\end{split}
\end{equation*}

(iii) This follows from the equality in \cref{eqn:lsc}.
\end{proof}

Note that although the value $\wass_c^{\T, R}(\mu, \nu)$ is stable under perturbation in $\mu$ and $\nu$, the optimizer $\theta_*$ is not necessarily continuous in $\mu$ and $\nu$. This is related to the question of uniqueness, which we consider next. In general, the minimizer of \cref{eq:wgencostpenal} is not unique. For example, the distribution $\mu$ may have some symmetry properties with respect to the set $\T$ of transformations. A more well-posed question is to ask whether the optimal $(T_\theta)_{\#}\mu$ is unique in $\T_{\#} \mu =\left\{ (T_\theta)_{\#}\mu \,: \theta \in \Theta\right\}$, although the same pushforward measure may be obtained by multiple transformations.  This question seems to be subtle. For example, when $c$ is the quadratic cost on Euclidean spaces, a natural guess is that this projection is unique when $\T_{\#} \mu$ is geodesically convex and the penalization function $R$ is geodescially convex on $\T_{\#} \mu$. But a counterexample can be found in \cite{adve25}.\footnote{We thank Daniel Lacker for pointing this paper to us.}  We also note that non-uniqueness and stability of solution is also a possible issue in other statistical methodologies such as principal component analysis. 

Here, we give a simple sufficient condition for uniqueness and leave a more complete analysis to future research. Let us restrict ourselves to Euclidean spaces $\mathcal{X}\subseteq \rr^n$ and $\mathcal{Z}\subseteq \rr^d$, and $c(y, z) = \|y - z\|^2$ is the quadratic cost. Note that $\mathcal{X}$ and $\mathcal{Z}$ do not necessarily have to be compact for the following discussion. Suppose also that there is a function $R$ on the $2$-Wasserstein space of $\rr^d$ such that, by an abuse of notation, we can write the penalization in the form $R(\theta)=R\left( (T_\theta)_\#\mu\right)$. We will use the concept of {\it generalized geodesic convexity} which can be found in \cite[Definition 7.31]{santam2015ot}.

\begin{proposition}\label{prop:uniqueness}
Consider the Euclidean set-up described above. Fix the target measure $\nu$ and let $\T_{\#} \mu$ be compact in the weak topology and generalized geodesically convex with base $\nu$.
Let $R$ be a proper, lower semicontinuous and strictly generalized geodesically convex function with base $\nu$. Then the penalized Wasserstein alignment problem \cref{eq:wgencostpenal} has a unique minimizer in $\T_{\#} \mu$, i.e., $\argmin_{\rho \in \Theta(\mu)} \left\{ \wass^2_2\left(  \rho, \nu \right) + R(\rho)\right\}$ is unique.
\end{proposition}

\begin{proof} It is known (\cite[Section 7.3, page 276]{santam2015ot}) that the function $\rho \mapsto \wass_2^2(\rho, \nu)$ is generalized geodesically convex with base $\nu$. Thus, the function $\rho \mapsto \wass_2^2(\rho, \nu)+ R(\rho)$ is lower semicontinuous and strictly generalized geodesically convex with base $\nu$. Thus it admits a unique minimum over a compact set.   
\end{proof}

\subsection{The Euclidean case} \label{sec:Euclidean}
The quadratic cost on Euclidean space and the associated $2$-Wasserstein distance are convenient and natural in many applications of optimal transport. In fact, the $2$-Wasserstein alignment between distributions on Euclidean spaces with unequal dimensions is the original motivation for this work.

In the following we describe precisely what we mean by the Euclidean case. Let $\X = \R^n$ and $\Z = \R^d$, where $n \geq d$. In this context, we call the source space $\X$ the {\it upward space} and the target space $\Z$ the {\it downward space}. We let $0_k$ be the zero vector in $\R^k$ and $I_k$ be the $k \times k$ identity matrix. All vectors are considered column vectors. By $\Prob_2(\R^n)$ we denote the space of all probability measures on $\R^n$ with finite second moment. We now replace the compactness condition in \cref{asmp:gendual}(i) by the finiteness of the first two moments of $\mu$ and $\nu$. 

\begin{assumption} \label{asmp:Euclidean}
We assume $\mu \in \Prob_2(\R^n)$ and $\nu \in \Prob_2(\R^d)$. 
\end{assumption}

Let $c(y, z) = \|y - z\|^2$ be the quadratic cost on the target space $\R^d$. Here, we write $c \equiv c_{\Z}$ to emphasize that it is the cost function on the target space. To define the {\it reverse alignment problem}, we also consider the Euclidean cost $c_{\X}(w, x) = \|w - x\|^2$ on the source space $\R^n$. Here we hide the dependence on the dimension ($n$ or $d$) of $\|\cdot\|$ which should be clear from the context.

The class of transformations we pick are orthogonal linear transformations indexed by 
\begin{equation} \label{eqn:Hplanes}
\Hplanes :=  \{ A \in \R^{n \times d} : A^{\top} A = I_d \}
\end{equation}
of $n \times d$ matrices with orthonormal columns. We regard $\Hplanes$ as a subset of $\R^{n \times d}$ and endow it with the distance induced by the Frobenius norm, also denoted by $\|\cdot \|$. For each $A \in \Hplanes$, $A^{\top}$, which is a $d \times n$ matrix, defines a linear map from $\R^n$ onto $\R^d$. Using the general formulation introduced in \cref{sec:intro}, we have $\Theta = \Hplanes$ and $T_{A}x := A^{\top} x$, for $A \in \Hplanes$ and $x \in \R^n$. If we define also the $S_A: \R^d \rightarrow \R^n$ by $S_Az := Az$, then, since $A^{\top} A = I_d$, we have that
\begin{equation} \label{eqn:orthogonal.projection}
S_A T_A x =  AA^{\top} x = \proj_{\mathrm{range}(A)} x
\end{equation}
is the orthogonal projection onto the column space of $A$. We let $A_{\#}^{\top} \mu \in \Prob(\R^d)$ be the pushforward of $\mu$ under the map $x \mapsto A^{\top} x$ (and similarly for $A_{\#} \nu \in \Prob(\R^n)$). It is helpful to think of $\X = \R^n$ as the state space of the observed data, and $\Z = \R^d$ as the latent space on which $\nu$ is taken to be a standard distribution (e.g.~Gaussian). Given $A \in \Hplanes$, we may think of $A_{\#} \nu$ as an approximation of $\mu$ based on the linear embedding $z \in \R^d \mapsto Az \in \R^n$. The transpose $x \mapsto A^{\top}x \in \Z$ provides an encoding map in the opposite direction. The identity~\cref{eqn:orthogonal.projection} implies that the composition $x \mapsto z = A^{\top}x \mapsto Az$ recovers $x$ if $x \in \mathrm{range}(A)$.

Following \cref{eqn:downward.Wasserstein}, we define the {\it downward $2$-Wasserstein loss} between $\mu$ and $\nu$ by
\begin{equation} \label{eqn:2.Wasserstein.downward}
\wass_2^{\downarrow}(\mu, \nu) := \wass_2^{\Hplanes}(\mu, \nu)  = \inf_{A \in \Hplanes} \wass_2 \left(  A_{\#}^{\top} \mu, \nu\right).
\end{equation}
There is no penalization term ($R \equiv 0$). In particular, if $n = d$, then \cref{eqn:2.Wasserstein.downward} amounts to finding an optimal orthogonal transformation to align $\mu$ with $\nu$ under $\wass_2$-loss. Since the $2$-Wasserstein distance is a metric, we may strengthen \cref{lem:existence} to a Lipschitz property without the compact support condition. With this, \cref{prop:existence} generalizes straightforwardly.

\begin{lemma} \label{lem:W2.estimates}
Let $(\mu, \nu)$ and $(\mu', \nu')$ be two pairs of distributions that satisfy Assumption \ref{asmp:Euclidean} and $A, A' \in \Hplanes$. Then, we have
\begin{equation} \label{eqn:stability}
|\wass_2( A_{\#}^{\top} \mu, \nu) - \wass_2( (A')_{\#}^{\top} \mu', \nu')| \leq c \|A - A'\| + \wass_2(\mu, \mu') + \wass_2(\nu, \nu')
\end{equation}
where $c:= \sqrt{\E[\norm{X}^2]}$ for $X \sim \mu$. 
\end{lemma}
\begin{proof}
Using the triangle inequality of $\wass_2$ (on $\Prob_2(\R^p)$  and $\Prob_2(\R^d)$), we have
\begin{equation*}
\begin{split}
&|\wass_2( A_{\#}^{\top} \mu, \nu) - \wass_2( (A')_{\#}^{\top} \mu', \nu')| \\
&\leq |\wass_2( A_{\#}^{\top} \mu, \nu) - \wass_2((A')_{\#}^{\top}\mu, \nu)| + |\wass_2((A')_{\#}^{\top}\mu, \nu) - \wass_2((A')_{\#}^{\top}\mu', \nu)| \\
&\quad + |\wass_2((A')_{\#}^{\top}\mu', \nu) - \wass_2((A')_{\#}^{\top}\mu', \nu')| \\
&\leq \wass_2( A_{\#}^{\top} \mu, (A')_{\#}^{\top}\mu) + \wass_2( (A')_{\#}^{\top} \mu, (A')_{\#}^{\top} \mu') + \wass_2(\nu, \nu').
\end{split}
\end{equation*}

To bound the first term, consider the coupling $(A^{\top} X, (A')^{\top}X)$ of $( A_{\#}^{\top} \mu, (A')_{\#}^{\top}\mu)$, where $X \sim \mu$. By the Cauchy--Schwarz inequality, we have 
\begin{equation} \label{eqn:AT.Lipschitz}
\wass_2^2( A_{\#}^{\top} \mu, (A')_{\#}^{\top}\mu) \le \E \norm{(A-A')^\top X}^2 \le \norm{A-A'}^2 \E[\norm{X}^2]. 
\end{equation}

For the second term, let $(X, X')$ be an optimal coupling for $\wass_2(\mu, \mu')$. Recall that $x \mapsto (A')(A')^{\top} x$ is the orthogonal projection onto the column space of $A'$ and hence is $1$-Lipschitz. Since $(A')^{\top}(A') = I_d$, we have
\begin{equation*}
\begin{split}
\wass_2( (A')_{\#}^{\top} \mu, (A')_{\#}^{\top} \mu') &\leq \E[ \| (A')^{\top} (X - X') \|^2 ]^{1/2} \\
&= \E[ (X - X')^{\top} (A')(A')^{\top} (X - X')]^{1/2} \\
&\leq \E[\|X - X'\|^2]^{1/2} = \wass_2(\mu, \mu').
\end{split}
\end{equation*}
Thus, the bound \cref{eqn:stability} has been proved.
\end{proof}

On the other hand, for each $A \in \Hplanes$, $A_{\#} \nu$ is a distribution on the upward space $\R^n$. Using the Euclidean cost $c_{\X}(x, w) := \|x - w\|^2$ on $\R^n$, we may define
\begin{equation*}
\begin{split}
\wass_2(\mu, A_{\#} \nu) &= \inf_{\gamma \in \Pi(\mu, \nu)} \left( \int c_{\X}(x, Az) d \gamma(x, z) \right)^{1/2},
\end{split}
\end{equation*}
where the use of $\Pi(\mu, \nu)$ (in place of $\Pi(\mu, A_{\#}\nu)$) follows from the argument of \cref{lem:existence}(i). This gives rise to the {\it upward alignment problem} with the upward $2$-Wasserstein loss
\begin{equation} \label{eqn:2.Wasserstein.upward}
\wass_2^{\uparrow}(\mu, \nu) := \inf_{A \in \Hplanes} \wass_2(\mu, A_{\#}\nu).
\end{equation}
This is another Wasserstein alignment problem where the roles of $\X$ and $\Z$ are reversed. This structure is specific to the Euclidean setting a matrix $A \in \mathbb{R}^{n \times d}$ defines two linear mappings $A: \R^n \rightarrow \R^d$ and $A^{\top}: \R^d \rightarrow \R^n$.

When dealing with high-dimensional data sets, it is common to normalize the data before performing statistical analysis. Under the following normalization condition, we will show that the downward and upward problems are, in fact, equivalent up to an additive constant.

\begin{assumption} \label{asmp:Euclidean2}
We assume that $\mu \in \Prob_2(\R^n)$ and $\nu \in \Prob_2(\R^d)$ have been normalized to have zero means and identity covariance matrices. Specifically, we assume
\begin{equation} \label{eqn:Euclidean.measure.normalization}
\begin{split}
&\int x \,d \mu(x) = 0_n, \quad \int xx^{\top} d\mu(x) = I_n,\\
&\int z \, d \nu(z) = 0_d, \quad \int zz^{\top} d\nu(z) = I_d.
\end{split}
\end{equation}
\end{assumption}

\begin{remark}
Our Euclidean setting is closely related to Procrustes alignment problem discussed in \cref{eg:Procrustes}, and the inner product Gromov--Wasserstein problem in \cref{eg:GW}. In  \cite{Alvarez2019} where the measures are assumed to be discrete, the normalization \eqref{eqn:Euclidean.measure.normalization} (also called {\it whitening}) is used to obtain an explicit update involving the singular value decomposition (see e.g.~their Lemma 4.2). In our paper, we consider general duality results that hold for general distributions that are not necessarily discrete. At our generality, it is unlikely to have explicit solutions/updates based on the singular value decomposition or related concepts.
\end{remark}


\begin{proposition}[Equivalence of downward and upward problems] \label{prop:Euclidean.equivalence}
Suppose \cref{asmp:Euclidean2} holds. Let $c_{\X}$ and $c_{\Z}$ be the Euclidean square costs on $\X = \R^n$ and $\Z = \R^d$ respectively. Then, for any $A \in \Hplanes$ and $\gamma \in \Pi(\mu, \nu)$, we have
\begin{equation} \label{eqn:equivalence}
\int c_{\X}(x, Az) d \gamma(x, z) = \int c_{\Z}(A^{\top} x, z) d \gamma(x, z) + (n - d).
\end{equation}
In particular, we have (denoting $\wass^{\downarrow, 2}_2 = ( \wass^{\downarrow}_2 )^2$ and similarly for $\wass^{\uparrow, 2}_2$)   
\[
\wass^{\uparrow, 2}_2(\mu, \nu) = \wass_2^{\downarrow, 2}(\mu, \nu) + (n - d),
\]
and $A \in \Hplanes$ is optimal for $\wass^{\uparrow}_2(\mu, \nu)$ if and only if it is optimal for $\wass_2^{\downarrow}(\mu, \nu)$.
\end{proposition}
\begin{proof}
Let $A \in \Hplanes$ and let $H = \mathrm{range}(A) \subset \R^n$ be the column space of $A$. Recall from \cref{eqn:orthogonal.projection} that $\proj_H x = AA^{\top} x$.
 For any $x\in \rr^p$ and $z \in \rr^d$, we have
\begin{equation*}
\begin{split}
c_{\X}(x, Az) &=\norm{x - Az}^2= \norm{x - \proj_H(x)}^2 + \norm{\proj_H(x) - Az}^2 \\
&= \norm{(I-AA^\top)x}^2 + \norm{AA^\top x - Az}^2.
\end{split}
\end{equation*}
Extend $A$ to an $n \times n$ orthonormal matrix $\bar{A} = \begin{bmatrix} A & B \end{bmatrix}$. Since norm does not change under a change of orthonormal basis, we have
\begin{equation*}
\begin{split}
c_{\X}(x, Az) &= \norm{(I-AA^\top)x}^2+ \norm{(\bar{A})^\top AA^\top x - (\bar{A})^\top Az}^2\\
&= x^\top (I-AA^\top) x + \norm{(\bar{A})^\top AA^\top x - (\bar{A})^\top Az}^2.
\end{split}
\end{equation*}
The last equality is due to the fact that $(I-AA^\top)$ is a projection matrix and, therefore, symmetric and idempotent. Now, both $AA^\top x$ and $Az$ are elements of $H$. Hence, they are both orthogonal to the last $n - d$ columns of $\bar{A}$. Therefore, 
\begin{equation} \label{eq:c-up-down}
c_{\X}(x, Az) = x^\top (I-AA^\top) x+ c_{\Z}(A^{\top}x,z). 
\end{equation}
The first term on the right depends only on $x$. From the normalization \cref{eqn:Euclidean.measure.normalization} and the identity $\Tr \left(A^\top A \right) = d$, for any $\gamma \in \Pi(\mu, \nu)$ we have
\[
\int x^\top (I-AA^\top) x \, d\gamma(x, y) = \int x^\top (I-AA^\top) x \,d \mu(x) = n-d, 
\]
This, in conjunction with~\cref{eq:c-up-down}, gives \cref{eqn:equivalence} from which the other statements are immediate.
\end{proof}

In \cref{sec:appendix.1} we illustrate the downward and upward alignment problems with a simulated example, showing that the optimization landscape can be highly non-convex.

\section{Generalized Kantorovich duality}
\label{sec:duality}
Consider the Wasserstein alignment problem
\begin{equation} \label{eqn:downward.problem.double.inf}
\wass_c^{\T, R}(\mu, \nu) = \inf_{\theta \in \Theta} \left\{ \inf_{\gamma \in \Pi(\mu, \nu)} \int_{\X \times \Z} c(T_{\theta}x, z) d \gamma(x, z) + R(\theta) \right\}
\end{equation}
as defined in \cref{sec:intro}. Here, we use \cref{lem:existence}(i) to optimize over couplings of $(\mu, \nu)$. In this section, we study the convex relaxation \cref{eq:convexify} under \cref{asmp:gendual} and use it to derive \cref{thm:duality} and \cref{cor:optimality}. We also prove \cref{thm:duality2} which is a refinement of \cref{thm:duality} using $\overline{c}$-concave functions.

\subsection{A convex relaxation}
To obtain a convex relaxation \cref{eqn:downward.problem.double.inf}, we allow $\theta$ to be randomized. In \cref{eq:convexify}, the key idea is to exploit independence. For clarity, we state the feasible set formally in a definition.

\begin{definition}[Feasible set for~\cref{eq:convexify}]
Given $\mu \in \Prob(\X)$ and $\nu  \in\Prob(\Z)$, define $\Upsilon(\mu, \nu)$ to be the set of Borel probability measures on $\X \times \Theta \times \Z$ such that if $(X, \eta, Z) \sim \gamma \in \Upsilon(\mu, \nu)$ then (i) $X \sim \mu$ and $Z \sim \nu$, and (ii) $X$ and $\eta$ are independent. 
\end{definition}

It is straightforward to verify that $\Upsilon(\mu, \nu)$ is a convex and closed subset of $\Prob(\X \times \Theta \times \Z)$. To check convexity, let $\gamma_1, \gamma_2 \in \Upsilon(\mu, \nu)$ and consider the mixture $\bar{\gamma}:=\frac{1}{2}\gamma_1 + \frac{1}{2}\gamma_2$. Obviously, under $\bar{\gamma}$, $X \sim \mu$ and $Z\sim \nu$. For the remaining requirement, let the pushforward of $\gamma_1$ by $(x,A,z) \mapsto (x,A)$ be $\mu \otimes P_1$ and the same for $\gamma_2$ be $\mu \otimes P_2$. Then, the same pushforward of $\bar{\gamma}$ is 
\[
\frac{1}{2} \mu \otimes P_1 + \frac{1}{2} \mu \otimes P_2 = \mu \otimes \left( \frac{1}{2} P_1 + \frac{1}{2} P_2  \right)= \mu \otimes \bar{P}.
\]
Since $\bar{P}$ is a probability on $\Hplanes$, we are done. To check closedness, suppose $\gamma_n \in \Upsilon(\mu, \nu)$ and $\gamma_n \rightarrow \gamma$ weakly. Clearly $\gamma$ satisfies (i). To check (ii), let $f \in C(\X)$ and $g \in C(\Z)$ (which are bounded by compactness of $\X$ and $\Z$). By weak convergence and independence of $X$ and $\eta$ under $\gamma_n$, we have
\[
\E_{\gamma_n}[f(X)g(\eta)] \rightarrow \E_{\gamma}[f(X)g(\eta)]
\]
and
\[
\E_{\gamma_n}[f(X)g(\eta)] = \E_{\gamma_n}[f(X)] \E_{\gamma_n}[g(\eta)] \rightarrow \E_{\gamma}[f(X)] \E_{\gamma}[g(\eta)].
\]
Since $f$ and $g$ are arbitrary, $X$ and $\eta$ are independent under $\gamma$. Hence $\eta \in \Upsilon(\mu, \nu)$.

Hence, \cref{eq:convexify} is an infinite-dimensional linear programming problem in the variable $\gamma$. A standard compactness argument using Prokhorov's theorem shows that \cref{eq:convexify} admits an optimal coupling. Note that if $\gamma_0 \in \Pi(\mu, \nu)$ and $\theta_0 \in \Theta$, we may define uniquely $\gamma \in \Upsilon(\mu, \nu)$ such that if $(X, \eta, Z) \sim \gamma$ then $(X, Z) \sim \gamma_0$ and $\eta = \theta_0$ almost surely.  Clearly,
\[
\int_{\X \times \Theta \times \Z} \left( c\left(T_{\theta}x, z\right) + R(\theta) \right) d \gamma(x, \theta, z) = \int_{\X \times \Z}  c(T_{\theta_0}x, z) \, d \gamma_0(x, z) + R(\theta_0).
\]
Thus \cref{eq:convexify} is indeed a convex relaxation of \cref{eqn:downward.problem.double.inf}. The following lemma shows that our relaxation does not change the optimal value. Moreover, its proof shows that there always exists an optimal $\gamma_* \in \Upsilon(\mu, \nu)$ for \cref{eqn:downward.problem.double.inf} under which $\eta$ is deterministic. The main advantage is that the convexity of \cref{eq:convexify} allows us to formulate a tractable dual problem.

\begin{lemma} \label{lem:Krelax}
Under \cref{asmp:gendual} we have $\overline{\wass}_c^{\T, R}(\mu, \nu) = \wass_c^{\T, R}(\mu, \nu)$.
\end{lemma}
\begin{proof}
From the discussion above, we have $\overline{\wass}_c^{\T, R}(\mu, \nu) \leq \wass_c^{\T, R}(\mu, \nu)$. To show the reverse inequality, let $\gamma_* \in \Upsilon(\mu, \nu)$ be optimal for \cref{eq:convexify}. Since the map
\[
\gamma \mapsto \int_{\X \times \Theta \times \Z} \left( c\left(T_{\theta}x, z\right) + R(\theta) \right) d\gamma
\]
is linear in $\gamma$, and $\Upsilon(\mu, \nu)$ is convex and compact, by the Krein--Milman Theorem we may choose $\gamma_*$ to be an extreme point of $\Upsilon(\mu, \nu)$. By an abuse of notations, regard $\eta$ as the coordinate map $\eta: \X \times \Theta \times \Z \rightarrow \Theta$, and let $P_* = \eta_{\#} \gamma_* \in \Prob(\Theta)$. 

We claim that $P_*$ is an extreme point of $\Prob(\Theta)$. For if not, we may write $P_* = \frac{1}{2} P + \frac{1}{2} P'$ for distinct $P, P' \in \Prob(\Theta)$. If $(X, \eta, Z) \sim \gamma_*$, let $\gamma_*(dxdz|\theta)$ denote the conditional distribution of $(X,Z)$ given $\eta=\theta$. Define $\gamma(dxd\theta dz)=\gamma_*(dxdz|\theta) P(d\theta)$ and $\gamma'(dxd\theta dz)=\gamma_*(dxdz|\theta) P'(d\theta)$. That is, under both $\gamma, \gamma'$, the conditional distribution of $(X,Z)$, given $\eta$, is the same as that of $\gamma_*$. 
Then it follows that $\gamma_* = \frac{1}{2} \gamma + \frac{1}{2} \gamma'$, and that $\gamma$ and $\gamma'$ are distinct elements of $\Upsilon(\mu, \nu)$. Thus, $\gamma_*$ cannot be an extreme point, contradicting our assumption. Since the above construction depends on regular conditional distributions, one needs to guarantee that the $P$ and $P'$ measure of the support of $P_*$ is one. But this follows from the fact that the supports of the measures $P$ and $P'$ must be subsets of the support of $P_*$. Since any set that has a positive measure under $P$ (or $P'$), then must also have a positive measure under $P_*=\frac{1}{2}P + \frac{1}{2}P'$. Hence the construction using conditional distributions is well defined.

Now, since $P_*$ is an extreme point of $\Prob(\Theta)$, it must be a point mass: $P_* = \delta_{\theta_*}$ for some $\theta_* \in \Theta$. Let $\pi_* \in \Pi(\mu, \nu)$ be the law of $(X, Z)$ if $(X, \eta, Z) \sim \gamma_*$, under which $\eta = \theta_*$ a.s. Then
\begin{equation*}
\begin{split}
\overline{\wass}_c^{\T, R}(\mu, \nu) &= \int_{\X \times \Theta \times \Z} \left( c\left(T_{\theta}x, z\right) + R(\theta) \right) d\gamma_* \\
&= \int_{\X \times \Z} \left( c(T_{\theta_*}x, z) + R(\theta_*)\right) d \pi_*\\
&\geq \wass_c\left( (T_{\theta_*})_{\#} \mu, \nu\right) + R(\theta_*) = \wass_c^{\T, R}(\mu, \nu),
\end{split}
\end{equation*}
which is the desired inequality. In particular, $\theta_*$ is optimal for $\wass_c^{\T, R}(\mu, \nu)$.
\end{proof}

By \cref{lem:Krelax}, we still write $\wass_c^{\T, R}(\mu, \nu)$ even when we optimize over $\Upsilon(\mu ,\nu)$ instead of $\Pi(\mu, \nu)$. The independence of $X$ and $\eta$ under $\gamma \in \Upsilon(\mu, \nu)$ will be exploited using the function space $\CEF_{\mu}$ given in Definition \ref{defn:whatisCEF}. We are now ready to prove \cref{thm:duality}.

\begin{proof}[Proof of \cref{thm:duality}]
The proof closely mirrors the argument for the usual Kantorovich duality; see, for example, the proof of \cite[Theorem 1.3]{Villani2003}. First, note that if $(\xi,\psi)\in \CEF_{\mu} \times C(\Z)$ satisfies \cref{eq:dualconstraint}, then for any $\gamma \in \Upsilon(\mu, \nu)$ and $\theta \in \Theta$ we have
\begin{equation*}
\begin{split}
\int_{\X} \xi(x, \cdot) d\mu(x) + \int_{\Z} \psi(z) d\nu(z)& =\int_{\X \times \Theta \times \Z} (\xi(x, \theta) + \psi(z)) d\gamma(x, \theta, z) \\
&\leq  \int_{\X \times \Theta \times \Z} \left( c\left( T_{\theta} x, z \right) + R(\theta) \right)d \gamma.
\end{split}
\end{equation*}
Here $\int_{\X \times \Theta \times \Z} \xi(x, \theta) d \gamma = \int \xi(x, \cdot) d\mu$ by independence of $X$ and $\eta$ under $\gamma$. Taking infimum over $\gamma \in \Upsilon(\mu, \nu)$ shows that weak duality holds, i.e., $\leq$ holds in \cref{eq:cxdual}.

The main part of the proof is to show that there is no duality gap. Let $\mathcal{M}_+$ denote the set of all nonnegative Borel measures on $\X \times \Theta \times \Z$. Let $\pi_1:\X \times \Theta \times \Z\rightarrow \X$, $\pi_2:\X \times \Theta \times \Z \rightarrow \Theta$ and $\pi_3:\X \times \Theta \times \Z \rightarrow \Z$ denote the three canonical coordinate projections. Then $\Upsilon := \Upsilon(\mu, \nu)$ is a convex subset of this convex cone $\mathcal{M}_+$. Let $\chi_{\Upsilon(\mu, \nu)}$ denote its characteristic function (also called the convex indicator function). That is, for any $\gamma \in \mathcal{M}_+$, $\chi_{\Upsilon}(\gamma) = \infty$ if $\gamma \in \Upsilon(\mu, \nu)$ and otherwise $\chi_{\Upsilon}(\gamma) = \infty$. Then, we may write
\begin{equation} \label{eqn:optim.with.indicators}
\begin{split}
\wass_c^{\T, R}(\mu, \nu) &= \inf_{\gamma \in \Upsilon(\mu, \nu)} \int\left(  c\left( T_{\theta} x, z\right) + R(\theta) \right) d\gamma \\
&= \inf_{\gamma \in \mathcal{M}_+} \left\{ \int \left( c\left( T_{\theta} x, z\right) + R(\theta) \right) d\gamma + \chi_\Upsilon(\gamma) \right\}.
\end{split}
\end{equation}
Observe that  $\Upsilon = \Upsilon_1\cap \Upsilon_2$, where $\Upsilon_1$ is the set of all probability measures $\gamma$ on $\X \times \Theta \times \Z$ such that $(\pi_1)_\#\gamma=\mu$ and $(\pi_3)_\#\gamma=\nu$, and $\Upsilon_2$ is the convex set of all probability measures $\gamma$ such that $(\pi_1, \pi_2)_\#\gamma$ is a product measure.

It is well known (see \cite[page 22]{Villani2003}) that for $\gamma \in \mathcal{M}_+$ we have
\begin{equation} \label{eq:dual.coupling}
\chi_{\Upsilon_1}(\gamma) = \sup_{\phi \in C(\X), \psi \in C(\Z)}\left\{ \int \phi(x) d\mu(x) + \int \psi(z) d\nu(z) - \int (\phi(x) + \psi(z))d\gamma \right\}.
\end{equation}
Note that the last integral only depends on $\gamma$ via $(\pi_1, \pi_3)_{\#} \gamma$.

Let $\CEF_{\mu, 0} := \left\{ \zeta \in \CEF_{\mu} : \int \zeta(x, \cdot) d\mu = 0 \right\}$. We now claim that 
\begin{equation}\label{eq:dualind}
\chi_{\Upsilon_2}(\gamma) = \sup_{ \zeta \in \CEF_{\mu, 0}}\left\{ \int \zeta(x, \theta) d\gamma(x, \theta, z)   \right\},
\end{equation}
where the integral depends on $\gamma$ via $(\pi_1, \pi_2)_{\#} \gamma$. To verify \cref{eq:dualind}, first suppose that $\gamma \in \Upsilon_2$. Then, by the independence of the first two coordinates, for any $\zeta \in \CEF_{\mu, 0}$ we have $\int \zeta(x, \theta) d\gamma = \int \zeta(x, \cdot) d\mu = 0$. Now, consider the complementary case $\gamma \notin \Upsilon_2$, under which the first two coordinates are dependent. 
We will demonstrate that there exists some $\zeta\in \CEF_{\mu, 0}$ such that $\int \zeta(x, \theta) d\gamma >0$.  Then, by multiplying that $\zeta$ by a sequence of positive constants that tends to $+\infty$, we have $\sup_{ \zeta \in \CEF_{\mu, 0}} \int \zeta d\gamma=\infty$. Together, we have \cref{eq:dualind}. Suppose $(X, \eta, Z) \sim \gamma$. Since $X$ and $\eta$ are dependent, there exist bounded measurable $\alpha$ on $\X$ and $\beta$ on $\Theta$ such that $\E_{\gamma} [\alpha(X)] = \int \alpha d\mu = 0$ but $\E_{\gamma}[\alpha(X) \beta(\theta)] > 0$. By a standard density argument, we may assume $\alpha$ and $\beta$ to be continuous. Then $\zeta(x, \theta) =\alpha(x) \beta(\theta)$ belongs to $\CEF_{\mu, 0}$ and satisfies the requirement. Substituting \cref{eq:dual.coupling} and \cref{eq:dualind} into \cref{eqn:optim.with.indicators}, we have
\begin{equation*}
\begin{split}
\wass_c^{\T, R}(\mu, \nu) 
&=\inf_{\gamma \in \mathcal{M}_+}\sup_{\phi, \psi, \zeta}\left\{ \int \left( c(T_{\theta}x,z) + R(\theta) \right) d\gamma + \right.\\
& \hspace{1cm} \left. + \int \phi d\mu + \int \psi d\nu  - \int (\phi(x) + \psi(z))\, d\gamma + \int \zeta  d\gamma \right\},
\end{split}
\end{equation*}
where the supremum is over $(\phi, \psi, \zeta)\in  C(\X) \times C(\Z) \times \CEF_{\mu, 0}$. 

Observe that if $(\phi, \zeta) \in C(\X) \times \CEF_{\mu, 0}$, then $\xi(x, \theta) = \phi(x) + \zeta(x, \theta) \in \CEF_{\mu}$ and $\int \xi(x, \cdot) d \mu = \int \phi d\mu$. Conversely, any $\xi \in \CEF_{\mu}$ can be written as $\xi = \int \xi(x, \cdot) d \mu + \left( \xi(x, \theta) - \int \xi(x, \cdot) d\mu \right)$ where the first term (as a constant function in $x$) is an element $C(\X)$ and the second term is an element of $\CEF_{\mu, 0}$. Thus, we may combine $\phi$ and $\zeta$ and express $\wass_c^{\downarrow, R}(\mu, \nu) $ as
\begin{equation}\label{eq:pfduality}
\begin{split}
\inf_{\gamma\in \mathcal{M}_+}\sup_{\xi, \psi }\left\{ \int  \xi d\mu + \int \psi d\nu - \int \left(  \xi(x,\theta) + \psi(z) - c(T_{\theta} x,z) - R(\theta) \right)d\gamma \right\},
\end{split}
\end{equation}
where now the supremum is over $\xi \in \CEF_{\mu}$ and $\psi \in C(\Z)$.

Now we will apply the Fenchel--Rockafellar duality \cite[Theorem 1.9]{Villani2003} to switch the $\inf$ and the $\sup$. For the convenience of the reader, we follow mostly the notation in the proof of \cite[Theorem 1.3]{Villani2003}. Consider the Banach space $E=C\left( \mathcal{X} \times \Theta \times \mathcal{Z} \right)$ be the Banach space of (bounded) continuous functions $u = u(x, \theta, z)$ on $\mathcal{X} \times \Theta \times \mathcal{Z}$ equipped with the usual supremum norm. Since $\mathcal{X} \times \Theta \times \mathcal{Z}$ is compact by assumption, its topological dual $E^*$ is the space of all (finite) signed measures on $\mathcal{X} \times \Theta \times \mathcal{Z}$ normed by total variation. For $u \in E$, define the convex functions 
\[
\Xi(u) :=\begin{cases}
    &\int \xi(x, \cdot) d\mu + \int \psi(z) d\nu, \quad \text{if } u= \xi(x, \theta) + \psi(z) \text{ for } \xi \in \CEF_{\mu} \text{ and } \psi \in  C(\Z); \\
    & \infty, \quad \text{otherwise}. 
\end{cases}
\]
Note that since the intersection of $\mathcal{F}_{\mu}$ and $C(\Z)$ (identified as subspaces of $E$) is the set of constant functions, the integral $\int \xi(x, \cdot) d\mu + \int \psi(z) d\nu$ is independent of the decomposition $u = \xi + \psi$. Also, define 
\[
\Theta(u) :=\begin{cases}
            0, &\quad \text{if}\; u(x,\theta,z) \ge -c\left(T_{\theta} x, z \right) - R(\theta) \text{ on $\X \times \Theta \times \Z$};\\
            \infty, & \quad \text{otherwise}.
        \end{cases}
\]
Then, 
\begin{equation} \label{eqn:Theta.plus.Xi}
\inf_{u \in E}\left\{ \Theta(u) + \Xi(u)\right\}= - \sup\left\{ \int \xi(x, \cdot) d\mu + \int \psi d\nu \right\},
\end{equation}
where the sup is over $(\xi, \psi) \in \mathcal{F}_{\mu} \times C(\Z)$ satisfying \cref{eq:dualconstraint}. To apply the Fenchel-Rockafellar duality theorem, we need the existence of some $u_0 \in E$ such that $\Theta(u_0) < \infty$, $\Xi(u_0) < \infty$ and that $\Theta$ is continuous at $u_0$. Let $M$ be the supremum of $-c(T_{\theta} x, z) - R(\theta)$ over $\X \times \Theta \times \Z$, which is finite by the compactness and continuity assumptions. We can simply take $u_0$ be the constant function $u_0 \equiv M + 1$. The duality theorem now gives
\begin{equation} \label{eqn:Fenchel.Rockafellar}
\inf_{u \in E} \left\{ \Theta(u) + \Xi(u) \right\} = \max_{\gamma \in E^*} \left\{ - \Theta^*(-\gamma) - \Xi^*(\gamma) \right\}.
\end{equation}

To finish the proof, we compute the Legendre-Fenchel conjugates of $\Theta$ and $\Xi$. For any measure $\gamma \in E^*$, it follows exactly as in \cite[page 27]{Villani2003} that
\[
\Theta^*(-\gamma)=\begin{cases}
&\int \left( c\left( T_{\theta} x, z\right) + R(\theta) \right) d\gamma, \quad \text{if $\gamma \in \mathcal{M}_+$,}\\
& \infty, \quad \text{otherwise}.
\end{cases}
\]
We now claim that $\Xi^*=\delta_\Upsilon = \delta_{\Upsilon_1} + \delta_{\Upsilon_2}$. To see this, note
\begin{equation*}
\begin{split}
\Xi^*(\gamma)&=\sup_{u \in E}\left\{ \int u d\gamma - \Xi(u) \right\}\\
&= \sup_{\xi \in \CEF_{\mu}, \psi \in C(\Z)}\left\{ \int \left(\xi(x,\theta) + \psi(z) \right)d\gamma - \int \xi(x,\cdot) d\mu - \int \psi(z)d\nu \right\}.
\end{split}
\end{equation*}
Letting $\xi = 0$ and varying $\psi$ shows that, in order for the above supremum to be finite, $(\pi_3)_\# \gamma =\nu$. By restricting to $\xi(x,\theta)\equiv \xi(x)$ gives us $(\pi_1)_\#\gamma=\mu$ for $\Xi^*(\gamma)$ to be finite. Finally, repeating the argument in the proof of \cref{eq:dualind} tells us that, unless $(\pi_1, \pi_2)_\#\gamma$ is a product measure, the supremum cannot be finite. Clearly, if $\gamma \in \chi_{\Upsilon}$ then the expression in side the bracket is $0$. So $\Xi^*=\delta_\Upsilon$.

Substituting \cref{eqn:Theta.plus.Xi} and the above into \cref{eqn:Fenchel.Rockafellar} and rearranging, we have
\begin{equation*}
\begin{split}
\sup \left\{ \int \xi(x, \cdot) d\mu + \int \psi d\nu \right\} &= \min_{\gamma \in E^*} \left\{ \Theta^*(-\gamma) + \Xi^*(\gamma) \right\} \\
&= \min_{\gamma \in \mathcal{M}_+} \left\{ \int \left( c(T_{\theta} x, z) + R(\theta) \right) d \gamma + \chi_{\Upsilon}(\gamma) \right\} \\
&= \wass_c^{\downarrow, R}(\mu, \nu).
\end{split}
\end{equation*}
\end{proof}

\subsection{Double convexification}
A more compact form of the generalized Kantorovich duality can be obtained by using $c$-duality. Specifically, we will apply the so-called ``double-convexification trick'' to \cref{thm:duality}. We continue to work under \cref{asmp:gendual}. Recall the spaces $\cconv(\Z)$ and $\cbarconv(\Z)$ given in Definition \ref{defn:cconcave}.

\begin{lemma} \label{lem:c.concave.regularity}
Under \cref{asmp:gendual}, both $\cconv(\Z)$ and $\cbarconv(\Z)$ are subsets of $C(\Z)$.
\end{lemma}
\begin{proof}
It is known (see \cite[Box 1.8, page 11]{santam2015ot}) that $c$-concave functions acquire the same modulus of continuity as $c$ which is uniformly continuous (also see the proof of \cref{lem:existence}). Since a $c$- or $\overline{c}$-concave function $\psi$ is finite at some point of $\mathcal{Z}$ by definition, it must be finite and uniformly continuous everywhere on $\mathcal{Z}$. 
\end{proof}

Recall from \cref{eqn:whatisIpsi} that $I_{\psi} : \Theta \rightarrow \R$ is defined for $\psi \in C(\Z)$ by
\[
I_\psi(\theta) = \int_{\X} \psi^{\overline{c}}(T_{\theta} x) d\mu(x) + R(\theta).
\]
By \cref{lem:c.concave.regularity}, $\psi^{\bar{c}}$ is (bounded and) uniformly continuous on $\Z$. From the continuity of $T_\theta$ and the bounded convergence theorem, we see that $I_{\psi} \in C(\Theta)$ and hence it attains its minimum $\min_{ \Theta} I_\psi$. Hence we may define $J_{\psi} \in C(\X \times \Theta)$ by
\begin{equation}\label{eq:jpsi}
J_\psi(x,\theta):=\psi^{\overline{c}}(T_{\theta} x) + R(\theta) - I_\psi(\theta) + \min_{ \Theta} I_\psi. 
\end{equation}
From the definition of $I_{\psi}$ we see that $J_{\psi} \in \CEF_{\mu}$; in fact, $\int J_\psi(x,\cdot)d\mu(x)= \min_{\Theta} I_\psi$.

Now, suppose $(\xi, \psi) \in \CEF_{\mu} \times C(\Z)$ satisfies the dual constraint \cref{eq:dualconstraint}:
\begin{equation}\label{eq:dualconstraint2}
\xi(x, \theta) \leq c\left(T_{\theta} x, z\right) + R(\theta) - \psi(z), \quad \mbox{for all} \;\; (x, \theta, z) \in \X \times \Theta \times \Z.
\end{equation} 
Taking infimum over $z \in \Z$ gives $\xi(x, \theta) \leq \psi^{\overline{c}}(T_{\theta}x) + R(\theta)$. From \cref{eq:jpsi} and \cref{eq:dualconstraint2}, we have
\[
J_{\psi}(x, \theta) \leq \psi^{\overline{c}}(T_{\theta}x) + R(\theta) \leq c(T_{\theta}x, z) + R(\theta) - \psi(z).
\]
Therefore, $(J_{\psi}, \psi)$ also satisfies the dual constraint. Moreover, integrating $\xi(x, \theta) \leq \psi^{\overline{c}}(T_{\theta}x) + R(\theta)$ over $\mu$ and minimizing over $\theta$ give
\[
\int \xi(x, \cdot) d\mu(x) \leq \min_{\theta \in \Theta} \int \left( \psi^{\overline{c}}(T_{\theta}x) + R(\theta) \right) d\mu(x) = \min_{\Theta} I_{\psi} = \int J_{\psi}(x, \cdot) d\mu(x).
\]
Thus, given $\psi \in C(\Z)$, we may let $\xi = J_{\psi} \in \CEF_{\mu}$ without decreasing the dual objective value. 

Next, we observe that $\psi$ itself can be taken to be a $\overline{c}$-concave function. For any $\psi \in C(\Z)$, we have that $\psi^{\overline{c}c} := (\psi^{\overline{c}})^c$ is $\bar{c}$-concave and satisfies $\psi^{\overline{c}c} \geq \psi$ and $\psi^{\overline{c}c\overline{c}} = \psi^{\overline{c}}$ (see the proof of \cite[Proposition 5.8]{V08}). Hence $J_{\psi^{\overline{c}c}} = J_{\psi}$ and
\[
\int J_{\psi}(x, \cdot) d \mu(x) + \int \psi d\nu \leq \int J_{\psi^{\overline{c}c}} (x, \cdot) d \mu(x) + \int \psi^{\overline{c}c} d\nu.
\]
Thus we only need to optimize over $\psi \in \cbarconv(\Z)$.

From the above discussion, for any $\mu \in \Prob(\X)$ and $\nu \in \Prob(\Z)$ we have
\begin{equation} \label{eqn:duality.sup.double.convex}
\begin{split}
\wass_c^{\downarrow, R}(\mu, \nu) &= \sup_{\psi \in \cbarconv(\Z)} \left\{ \int J_\psi(x, \cdot) d\mu  + \int \psi(z)d\nu  \right\} \\
 &= \sup_{\psi \in \cbarconv(\Z)} \left\{ \min_{\Theta} I_{\psi} + \int \psi(z) d \nu \right\}.
\end{split}
\end{equation}
On the other hand, for any $\theta \in \Theta$, the usual Kantorovich duality (see \cite[Proposition 1.11]{santam2015ot}) for the optimal transport problem $\wass_c( (T_{\theta})_{\#} \mu, \nu)$ between $(T_{\theta})_{\#} \mu$ and $\nu$ with cost $c$ gives
\begin{equation} \label{eqn:usual.Kantorovich}
\begin{split}
&\wass_c( (T_{\theta})_{\#} \mu, \nu) + R(\theta)\\
&= \max_{\psi  \in \cbarconv(\Z)} \left\{ \int \psi^{\overline{c}}(y) d ((T_{\theta})_{\#}\mu)(y) + \int \psi(z) d \nu(z) \right\} + R(\theta)\\
  &=  \max_{\psi  \in \cbarconv(\Z)} \left\{ I_{\psi}(\theta) + \int \psi(z) d \nu(z) \right\},
\end{split}
\end{equation}
since the term $R(\theta)$ is independent of the choice of $\psi$. Note that in \cref{eqn:usual.Kantorovich} the maximum is attained. For $\theta \in \Theta$ fixed, we call $(\psi_*^{\overline{c}}, \psi_*)$ a pair of {\it Kantorovich potentials} for $\wass_c( (T_{\theta})_{\#} \mu, \nu)$ if $\psi_* \in \cbarconv(\Z)$ is optimal for \cref{eqn:usual.Kantorovich}. Hence, from \cref{eq:wgencost} we also have
\begin{equation} \label{eqn:usual.Kantorovich2}
\begin{split}
\wass_c^{\T, R}(\mu, \nu) = \min_{\theta \in \Theta}     \max_{\psi  \in \cbarconv(\Z)} \left\{ I_{\psi}(\theta) + \int \psi(z) d \nu(z) \right\}.
\end{split}
\end{equation}
The following theorem shows that the supremum in \cref{eqn:duality.sup.double.convex} is a maximum, that is, the order of the min over $\theta$ and the max over $\psi$ in \cref{eqn:usual.Kantorovich2} can be switched. 

\begin{theorem}[Duality with $\overline{c}$-concave functions] \label{thm:duality2}
Under \cref{asmp:gendual}, we have
\begin{equation}\label{eq:cxdual2}
            \wass_c^{\T, R}(\mu, \nu)= \max_{\psi \in \cbarconv(\Z)} \left\{ \min_\Theta I_\psi + \int \psi(z)d\nu(z)  \right\},
\end{equation}
where the maximum is attained. 
\end{theorem}

\begin{proof}
From the discussion above, \cref{thm:duality} implies \cref{eqn:duality.sup.double.convex} which has a supremum over $\psi \in  \cbarconv(\Z)$. It remains to show that the supremum is attained. We give only a sketch here since the argument is well-established, see e.g.,~the proof of \cite[Proposition 1.1]{santam2015ot}. Let $(\psi_k) \subset \cbarconv(\Z)$ be a maximizing sequence, i.e., $\lim_{k \rightarrow \infty} \left( \int J_{\psi_k}(x, \cdot) d\mu + \int \psi d\nu\right) = \wass_c^{\T, R}(\mu, \nu)$. Since each $\psi_k$ shares the same modulus of continuity as that of $c$ (see the proof of \cref{lem:c.concave.regularity}), the sequence $(\psi_k)$ is equicontinuous. By shifting each $\psi_k$ by a constant, we may assume that $(\psi_k)$ is uniformly bounded. By the Arzela--Ascoli theorem, we can assume, at least through a subsequence, that $\lim_{k\rightarrow \infty}\psi_k = \psi_*$ uniformly for some $c$-concave function $\psi_*$. A similar argument (see the proof of \cite[Proposition 1.1]{santam2015ot}) shows that, by passing through a further subsequence, we have $\lim_{k\rightarrow \infty}\psi_k^{\overline{c}} = \psi_*^{\overline{c}}$ uniformly as well. It follows that $\lim_{k \rightarrow \infty} I_{\psi_k}(\cdot) = I_{\psi_{*}}(\cdot)$ uniformly and $\lim_{k \rightarrow \infty} \min_{\Theta} I_{\psi_k}(\cdot) = \min_{\Theta} I_{\psi_*}(\cdot)$. Thus, $\lim_{k \rightarrow \infty} J_{\psi_k} = J_{\psi_*}$ uniformly. It follows from the uniformly convergence that $\wass_c^{\downarrow, R} (\mu, \nu) = \int J_{\psi_*}(x, \cdot) d\mu + \int \psi_* d\mu$. So, the supremum is attained by $\psi_*$. 
\end{proof}

\begin{proof}[Proof of \cref{cor:optimality}]
We first prove the second inequalities in \cref{eqn:gap} which also proves the \textit{if} part of the optimality condition (the first inequality is immediate). Since $(\psi^{\overline{c}}, \psi)$ is a pair of Kantorovich potentials for $\wass_c\left( (T_{\theta})_{\#}\mu, \nu \right)$, we have
\begin{equation*}
\begin{split}
\wass_c\left( (T_{\theta})_{\#} \mu, \nu \right) + R(\theta) 
&= \int \psi^{\overline{c}}(y) d ((T_{\theta})_{\#} \mu)(y) + \int \psi(z) d \nu(z) + R(\theta)\\
&= \int \psi^{\overline{c}}\left(T_{\theta} x\right) d\mu(x) + R(\theta) + \int \psi(z)d\nu(z) \\
&= I_{\psi}(\theta) + \int \psi(z) d \nu(z) \\
&= \int J_{\psi}(x, \cdot) d\mu(x) + \int \psi(z) d \nu(z) + \left( I_{\psi}(\theta) -\min_{\Theta} I_{\psi}\right) \\
&\leq \wass_c^{\T, R}(\mu, \nu) + \left( I_{\psi}(\theta) -\min_{\Theta} I_{\psi}\right),
\end{split}
\end{equation*}
where the last inequality follows from \cref{thm:duality2}. 

We now argue the \textit{only if} part of the condition for optimality. From \cref{eqn:usual.Kantorovich2} and \cref{eq:cxdual2} it follows that $\wass_c^{\T, R}(\mu, \nu)$ is equal to
\[
 \min_{\theta \in \Theta}     \max_{\psi  \in \cbarconv(\Z)} \left\{ I_{\psi}(\theta) + \int \psi(z) d \nu(z) \right\} = \max_{\psi \in \cbarconv(\Z)} \left\{ \min_\Theta I_\psi + \int \psi(z)d\nu  \right\}.
\]
Let $\theta_*$ be optimal for $\wass_c^{\T, R}(\mu, \nu)$ and let $\psi_* \in \cbarconv(\Z)$ be optimal for the RHS. Then
\begin{equation}\label{eq:corpf2}
  \max_{\psi  \in \cbarconv(\Z)} \left\{ \int \psi^{\bar{c}}\left( T_{\theta_*}x \right)d \mu(x) + R(\theta) +  \int \psi(z) d \nu(z) \right\}= \min_\Theta I_{\psi_*} + \int \psi_*(z)d\nu. 
\end{equation}
Take $\psi=\psi_*$ in the LHS to get 
\[
 I_{\psi_*}(\theta_*)=\int \psi_*^{\bar{c}}\left( T_{\theta_*}x \right)d \mu(x) + R(\theta) \le \min_\Theta I_{\psi^*}.
\]
Since the reverse inequality is trivial, there must be equality above. Thus $\theta_* \in \argmin_\Theta I_{\psi^*}$.

Substituting in \cref{eq:corpf2} we get 
\begin{equation*}
\begin{split}
&\wass_c\left( (T_{\theta_*})_{\#} \mu, \nu \right) + R(\theta_*) = \wass_c^{\downarrow, R}(\mu, \nu)\\ 
&= \int \psi_*^{\bar{c}}\left( y \right)d \left((T_{\theta_*})_\#\mu \right)(y) + R(\theta_*) + \int \psi_*(z)d\nu(z).
\end{split}
\end{equation*}
Thus $\left( \psi_*, \psi_*^{\bar{c}}\right)$ is a pair of Kantorovich potentials for $\wass_c\left( (T_{\theta_*})_{\#}\mu, \nu \right)$. 
\end{proof}

Let us derive some consequences of \cref{eq:cxdual2}. Let $\psi_0 \in \cbarconv(\Z)$ and let $\theta_0 \in \Theta$. Let $G(\cdot)$ denote the optimality gap function for the dual objective function, i.e.
\[
G(\psi_0) := \max_{\psi \in \cbarconv(\Z)} \left\{ \min_\Theta I_\psi + \int \psi(z)d\nu  \right\} - \left( \min_\Theta I_{\psi_0} + \int \psi_0(z)d\nu  \right) \ge 0. 
\]
Also, consider the primal optimality gap function 
\[
\Delta(\theta_0):= \wass_c( (T_{\theta_0})_{\#} \mu, \nu) + R(\theta_0) - \wass_c^{\downarrow, R}(\mu, \nu) \ge 0.
\]
From \cref{eq:cxdual2}, it follows after some algebraic manipulations, 
\begin{equation*}
\begin{split}
&\wass_c( (T_{\theta_0})_{\#} \mu, \nu) + R(\theta_0) - \Delta(\theta_0) = \left( \min_\Theta I_{\psi_0} + \int \psi_0(z)d\nu  \right) + G(\psi_0),\quad \text{or}, \\
&\Delta(\theta_0) + G(\psi_0) = \left( \wass_c( (T_{\theta_0})_{\#} \mu, \nu) + R(\theta_0) \right) - \left( \min_\Theta I_{\psi_0} + \int \psi_0(z)d\nu  \right).
\end{split}
\end{equation*}
Let $\psi_0^*$ denote some $\overline{c}$-concave Kantorovich potential for $\wass_c( (T_{\theta_0})_{\#} \mu, \nu)$. Using a shift, if necessary, we can guarantee that $\int \psi^*_0(z)d\nu = \int \psi_0(z)d\nu$. It follows that
\begin{equation*}
\begin{split}
\Delta(\theta_0) + G(\psi_0) &= \int (\psi_0^*)^{\overline{c}} (T_{\theta_0}x) d\mu(x) + \int \psi_0^* d \nu + R(\theta_0) - \min_{\Theta} I_{\psi_0} - \int \psi_0 d \nu\\
&=  I_{\psi^*_0}(\theta_0)  -  \min_{\Theta} I_{\psi_0}.
\end{split}
\end{equation*}
Since the LHS is nonnegative, so must be the RHS. In particular, if $\psi_0=\psi_0^*$ we obtain a strengthening of \cref{eqn:gap} given by 
\begin{equation}\label{eqn:gap2}
\Delta(\theta_0) + G(\psi^*_0) =  I_{\psi^*_0}(\theta_0)  -  \min_{\Theta} I_{\psi^*_0}.
\end{equation}

\section{Orthogonal transformations of Euclidean spaces} \label{sec:Euclidean.further}

Consider the downward alignment problem~\cref{eqn:2.Wasserstein.downward} under the Euclidean setting in \cref{sec:Euclidean}. We continue to work under \cref{asmp:Euclidean}. We already know that a global minimizer of the map $A\in \Hplanes \mapsto \wass_2^2(A^\top_\#\mu, \nu)$ exists, thanks to \cref{prop:existence} (see the paragraph above \cref{lem:W2.estimates}), but there may be local minimizers as seen in the simulated example in \cref{sec:Euclidean}. 

We give a first-order condition for local minimizers of~\cref{eqn:2.Wasserstein.downward} using the structure of $\Hplanes$. For technical purposes, we will now assume, in addition to \cref{asmp:Euclidean}, that $\nu$ is absolutely continuous with respect to the Lebesgue measure, denoted by $\nu \in \Prob_{2, ac}(\R^d)$. By Brenier's theorem (see \cite[Section 6.2.3]{AGS05} for a general statement true for Hilbert spaces), for every choice of $A\in \Hplanes$ there is a $\nu$-a.e.~convex gradient $\nabla F_A$, where $F_A: \rr^d\rightarrow \rr\cup \{+\infty\}$ is convex and lower semicontinuous, such that the coupling $(\nabla F_A (Z), Z)$, $Z \sim \nu$, is optimal for $\wass_2(A_{\#}^{\top} \mu, \nu)$. We call $\nabla F_A$ the Brenier map from $\nu$ to $A_{\#}^{\top} \mu$. To extend it to a coupling of $(X,Z)$, require that the conditional distribution of $X$, given $Z=z$, be the same as the conditional distribution, under $\mu$, of $X$ given $A^\top X= \nabla F_A(z)$. Clearly, this preserves the marginal distribution of $X$. We will refer to the coupling constructed this way as the {\it optimal coupling} of $(\mu,\nu)$ for $A$. 

\begin{theorem}[First order condition]\label{thm:downfoc}
Let $\mu \in \Prob_2(\R^n)$ and $\nu \in \Prob_{2, ac}(\R^d)$. Let $A_*$ be a local minimizer of the map $A\in \Hplanes \mapsto \wass_2^2(A^\top_\#\mu, \nu)$ such that $(A_*^{\top})_{\#} \mu \in \Prob_{2,ac}(\R^d)$, and let $\nabla F_{A_*}(\cdot)$ be the $\nu$-a.s.~unique Brenier map pushforwarding $\nu$ to $(A_*^\top)_\# \mu$. Then the following cross-correlation constraint holds: 
\begin{equation} \label{eqn:covariance.condition}
\E\left[ \nabla_i F_{A_*}(Z) Z_j \right]= \E\left[ \nabla_j F_{A_*}(Z) Z_i \right], \quad 1 \leq i, j \leq d,
\end{equation}
where $(X, Z)$ are optimally coupled for $A_*$. In particular, if \cref{asmp:Euclidean2} holds, then under the same coupling we have
\begin{equation} \label{eqn:cross.correlation}
\Cor( (A_*^{\top} X)_i, Z_j ) = \Cor( (A_*^{\top} X)_j, Z_i ), \quad 1 \leq i, j \leq d.
\end{equation}
\end{theorem}

\begin{proof}
Let $A_* \in \Hplanes$ be a local minimum of the map $A\in \Hplanes \mapsto \wass_2^2(A^\top_\#\mu, \nu)$. We begin by constructing smooth perturbations of $A_*$ in $\Hplanes$. Let $\cO(n) := \{ \bar{A} \in \R^{n \times n} : \bar{A} \bar{A}^{\top} = I_n\}$ denote the set of all real $n\times n$ orthogonal matrices. For $\bar{A} \in \cO(n)$, we write $\bar{A} = \begin{bmatrix} A & A'\end{bmatrix}$ where $A \in \Hplanes$ and $A' \in \R^{n \times (n - d)}$. Given $A_*$, let $\bar{A}_* = \begin{bmatrix} A_* & A_*' \end{bmatrix} \in \cO(n)$ be such that its first $d$ columns coincide with $A_*$. Let $S \in \R^{n \times n}$ be skew-symmetric, i.e., $S^{\top} = -S$. For $t \in \R$, let $\bar{A}_t := e^{tS} \bar{A}_*$ and note that $\bar{A}_t \in \cO(n)$.\footnote{Recall that a fundamental property of skew-symmetric matrices is that their exponentials are orthogonal matrices with determinant 1.} Let $A_t := e^{tS} A_*$ be the first $d$ columns of $\bar{A}_t$. Then $(A_t)_{t \in \R}$ is a curve in $\Hplanes$ with $A_0 = A_*$. 

Consider a skew-symmetric $S \in \R^{n \times n}$ of the form $S = A_* C A_*^{\top}$, where $C \in \R^{d \times d}$ is skew-symmetric. Observe that $e^{tS} = A_* e^{tC} A_*^{\top}$ since $A_*^{\top} A_* = I_d$. With this $S$ we have $A_t = A_* e^{tC}$. For $t \in \R$, define $\rho_t := (A_t^{\top})_{\#} \mu = (e^{-tC} A_*^{\top})_{\#} \mu$. That is, each particle travels along the curve $y_t = e^{-tC} y_0$. It is easy to see that $(\rho_t)_{t \in \R}$ is an absolutely continuous curve in $(\Prob_{2, ac}(\R^d), \wass_2)$ with $\rho_0 = (A_*^{\top})_{\#} \mu$. Moreover, $(\rho_t)$ satisfies the continuity equation $\dot{\rho}_t + \nabla \cdot (v_t \rho_t) =0$ with $v_0(y) = -Cy$. Using the optimality of $A_*$ and \cite[Corollary 10.2.7]{AGS05} (this requires absolute continuity of $(A_*^{\top})_{\#} \mu$), we have
\begin{equation} \label{eqn:first.order.condition.pxp}
\begin{split}
0 &= \left. \frac{d}{d t} \right|_{t = 0} \frac{1}{2} \wass_2^2(\rho_t, \nu) = \int v_0( \nabla F_{A_*}(z)) \cdot (\nabla F_{A_*}(z) - z) \, d\nu(z) \\
&= \E \left[ (-C \nabla F_{A_*}(Z)) \cdot (\nabla F_{A_*}(Z) - Z) \right].
\end{split}
\end{equation}
Rearranging and using skew-symmetry, we have the following first order condition for any skew-symmetric $C \in \R^{d \times d}$:
\begin{equation} \label{eqn:first.order.condition.dxd}
\E\left[ (\nabla F_{A_*}(Z))^{\top} C  (\nabla F_{A_*}(Z) - Z)\right] = 0.
\end{equation}
For $1 \leq i, j \leq d$, let $C$ be the skew-symmetric matrix $C = e_i e_j^{\top} - e_j e_i^{\top} \in \R^{d \times d}$ where $(e_i)_{1 \leq i \leq d}$ is the standard basis of $\R^d$. Plugging this $C$ into \cref{eqn:first.order.condition.dxd} gives the first-order condition \cref{eqn:covariance.condition}.

Write $\nabla F_{A_*}(Z) = A_*^{\top} X$ where $(X, Z)$ are optimally coupled for $A_*$. Since $A_*^{\top} A_* = I_d$, $A_*^{\top} X$ (and also $Z$) has mean $0_d$ and covariance matrix $I_d$ under \cref{asmp:Euclidean2}. From this and \cref{eqn:covariance.condition} we obtain the cross-correlation condition \cref{eqn:cross.correlation}.
\end{proof}


\subsection{An explicit example} \label{sec:explicit.example}
In this subsection, we analyze an explicit example with $n = 2$ and $d = 1$ to illustrate the first-order condition \cref{eqn:covariance.condition} in \cref{thm:downfoc} as well as the optimality condition in \cref{cor:optimality}.

Let $\nu = N(0,1)$ be the one-dimensional standard normal distribution on $\R$. For $a\in \rr^2 \setminus \{0\}$ fixed, let $\mu$ be the normal mixture distribution $\frac{1}{2} N(a, I_2) + \frac{1}{2} N(-a, I_2)$. That is, if $X \sim \mu$, one may simulate it by first tossing a fair coin $\varepsilon \sim \text{Bernoulli}(1/2)$; then if $\varepsilon=+1$ then sampling from $N(a, I_2)$; otherwise, if $\varepsilon=0$ sampling from $N(-a, I_2)$.\footnote{Here $\mu$ and $\nu$ are not compactly supported and so \cref{cor:optimality} does not apply directly. Nevertheless, we will show that its conclusion holds in this case.} 

The elements of $\Hplanes$ are given by column unit vectors of the form  $\lambda = \begin{bmatrix} \cos(\theta) & \sin(\theta)\end{bmatrix}^\top$, $\quad \theta \in [0, 2\pi)$. For $\lambda \in \Hplanes$, we have
\begin{equation} \label{eqn:normal.mixture}
 \lambda_{\#}^\top\mu = \frac{1}{2} N\left( \lambda^\top a, 1\right) + \frac{1}{2} N\left( -\lambda^\top a, 1\right),
\end{equation}
which is again a normal mixture. If $\lambda^{\top} a = 0$, then $\lambda^{\top}_{\#}\mu = N(0, 1) = \nu$ and so $\wass_2( \lambda_{\#}^{\top}\mu, \nu) = 0$. Otherwise, we have $\lambda_{\#}^{\top} \mu \neq \nu$ and the Wasserstein distance is positive. Thus, $\wass_2^{\downarrow}(\mu, \nu) = 0$ and the set $H_a$ of optimal $\lambda \in \Hplanes$ consists of the two unit vectors orthogonal to $a$.

We first verify that the first-order condition \cref{eqn:covariance.condition} in \cref{thm:downfoc} holds. For $\lambda \in H_a$, the Brenier map from $\nu$ to $\lambda_{\#}^{\top} \mu = \nu$ is the identity map, that this, $\nabla F_\lambda(z) = F_\lambda'(z) z$. Since $d=1$, $i = j = 1$ and \cref{eqn:covariance.condition} reduces to $\E[Z^2]=\E [Z^2]$, which is obviously true.  

For \cref{cor:optimality}, the argument is more involved but may have some independent interest. Without loss of generality, we use the downward cost $c(y, z) = \frac{1}{2}(y - z)^2$. From the symmetry of $c$ we have $\psi^c = \psi^{\overline{c}}$ and $\cconv(\R) = \cbarconv(\R)$, so that $I_{\psi}(\lambda) = \int \psi^c(\lambda^{\top}x) d\mu(x)$ when $\psi$ is $c$-concave. From \cite[Proposition 1.21]{santam2015ot}, $\psi(z)$ is $c$-concave if and only if $\frac{1}{2}z^2 - \psi(z)$ is convex and lower semicontinuous, and $\psi^c(y) = \frac{1}{2}y^2 - \psi^*(y)$ where $\psi^*$ is the convex conjugate of $\psi$. 

First, let $\lambda \in H_a$ be optimal, so that $\lambda_{\#}^{\top} \mu = \nu$. Then $(\psi^c, \psi) = (0, 0)$ is a pair of Kantorovich potentials. Clearly, we have $I_{\psi}(\lambda) \equiv 0$ and, trivially, $\lambda \in \argmin_{\Hplanes}  I_{\psi}$. 

Next, consider $\lambda \notin H_a$, and let $(\psi^c, \psi)$ be a pair of Kantorovich potentials for $\wass_2(\lambda_{\#}^{\top} \mu, \nu)$. Since $\mu$ and $\nu$ are fully supported, from Brenier's theorem $\psi^c$ is unique up to an additive constant. To verify that $\lambda \notin \argmin_{\Hplanes} I_{\psi}$, we need the following claim whose (elementary but somewhat lengthy) proof can be found in \cref{sec:appendix.2}. 

\begin{proposition} \label{prop:psiclaim}
The function $\psi^c$ is strictly convex. 
\end{proposition}



Given \cref{prop:psiclaim}, the rest of the argument proceeds as follows. Since $\psi^c$ is strictly convex, the function $\rho \mapsto \int \psi^c(y) d\rho(y)$ is strictly displacement convex on $\Prob_{2, ac}(\R)$ with respect to McCann's displacement interpolation \cite[Theorem 5.15]{Villani2003}. For $b \in \Hplanes$, let $\rho_0=N(b^{\top} a, 1)$ and $\rho_1=N(-b^{\top} a, 1)$. Then the McCann displacement interpolation between $\rho_0$ and $\rho_1$ at time $\frac{1}{2}$ is precisely $\rho_{1/2}=N(0,1)$. Thus, from \cref{eqn:normal.mixture} and geodesic convexity, we have
\begin{equation*}
\begin{split}
\int \psi^c(\lambda^\top x) d\mu(x) = \frac{1}{2} \int \psi^c(y)d\rho_0(y) + \frac{1}{2} \int \psi^c(y) d\rho_1(y) \geq \int \psi^c(y) d\rho_{1/2}(y) 
\end{split}
\end{equation*}
If $b = \lambda$ (so that $\lambda^{\top}a \neq 0$), then $\rho_0 \neq \rho_1$ and the inequality is strict. 
On the other hand, if $b = \lambda_* \in H_a$ then $\rho_0 = \rho_1$ and equality holds. It follows that
\[
I_{\psi}(\lambda) = \int \psi^c(\lambda^{\top}x) d \mu(x) > \int  \psi^c(y) d \rho_{1/2}(y) = \int \psi^c(\lambda_*^{\top} x) d\mu(x) > I_{\psi}(\lambda_*).
\]
This proves that $\lambda \notin \argmin_{\Hplanes} I_{\psi}$.

\appendix

\section*{Appendix}
\section{Further examples} \label{sec:further.examples}

\begin{example}[Encoder and decoder perspective] 
	Consider $\X\subset \R^n$ and $\Z \subset \R^d$, and $T_\theta: \R^n \to \R^d$ be a neural network parameterized by $\theta \in \Theta \subset \R^k$. In the framework of variational autoencoders (VAEs)~\cite{VAE, VAEintro}, we can interpret $T_\theta$ as an {\it encoder} when $n \gg d$, mapping high-dimensional input data to a lower-dimensional latent space. In contrast, when $d \gg n$, $T_\theta$ can be viewed as a {\it decoder}, transforming low-dimensional latent representations into high-dimensional outputs.
	
	Unlike VAEs, which typically rely on the Kullback–Leibler (KL) divergence (or its variational approximation) to train the neural network, our formulation in~\cref{eq:wgencost} employs the Wasserstein distance to guide the training of $T_\theta$. This approach aligns with Wasserstein-based generative models~\cite{Bousquet2017, Tolstikhin2017}, which have been shown to offer superior stability and more meaningful latent representations by leveraging optimal transport principles.
\end{example}

There are other applications where we believe our techniques might be useful. One such application is the SE(3)-Transformer~\cite{Fuchs-2020} which is a specialized deep learning architecture designed for 3D geometric data, ensuring that predictions remain equivariant under rigid transformations (rotations and translations). It extends the transformer architecture to handle spatially structured data, such as molecular structures, protein folding, point clouds, and physical simulations. For example, DeepMind's AlphaFold 2 uses SE(3)-equivariant neural networks to refine protein structures. This ensures that the predicted protein conformation does not depend on the coordinate system.
We believe that our method can be used to address alignment problems~\cite{Bryant-2022, Nguyen-2024} arising in this context.

\section{Simulated example of the Euclidean case} \label{sec:appendix.1}
\begin{figure}[t!]
	\centering
	\includegraphics[scale = 0.55]{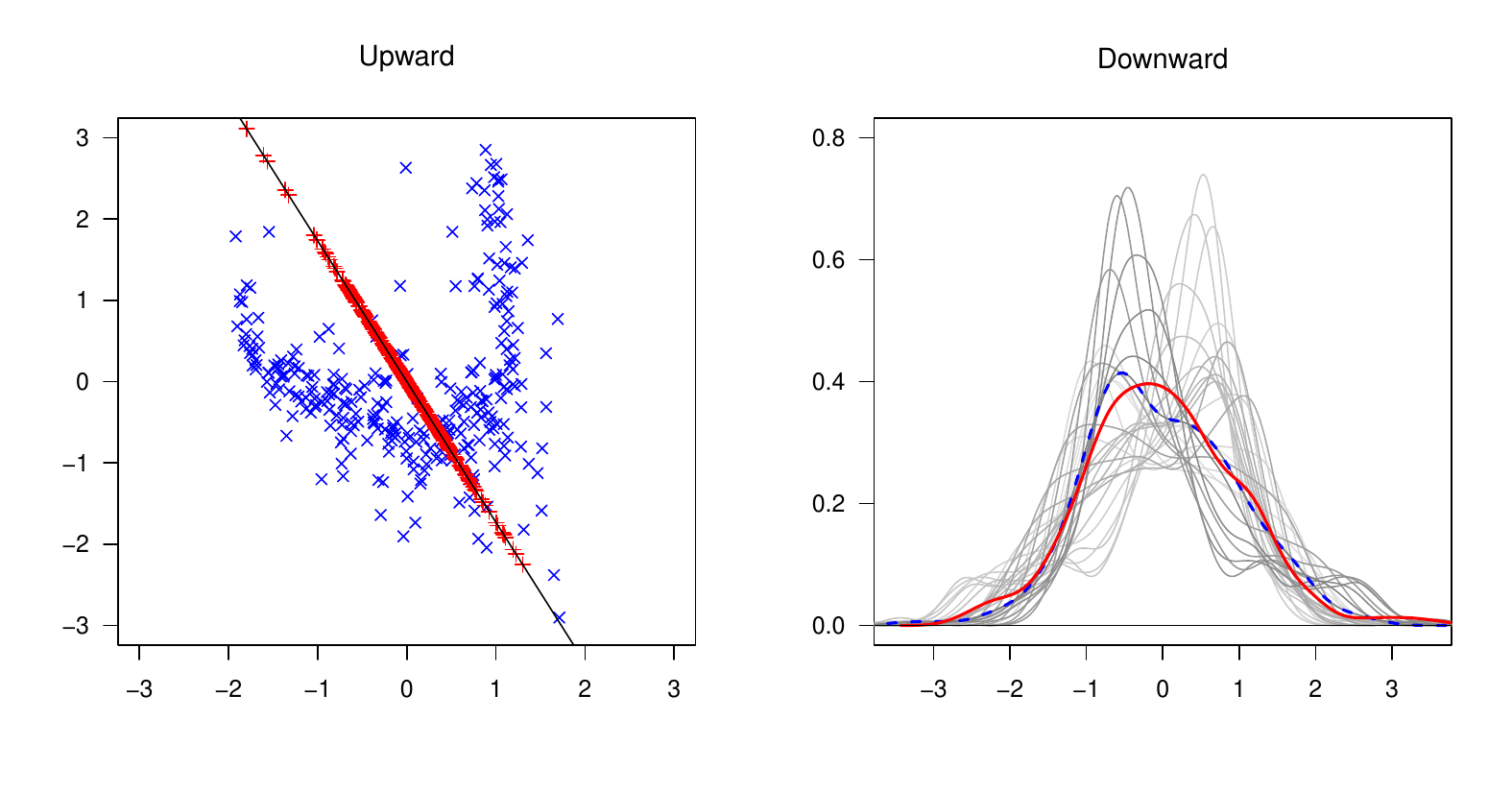}
	\includegraphics[scale = 0.55]{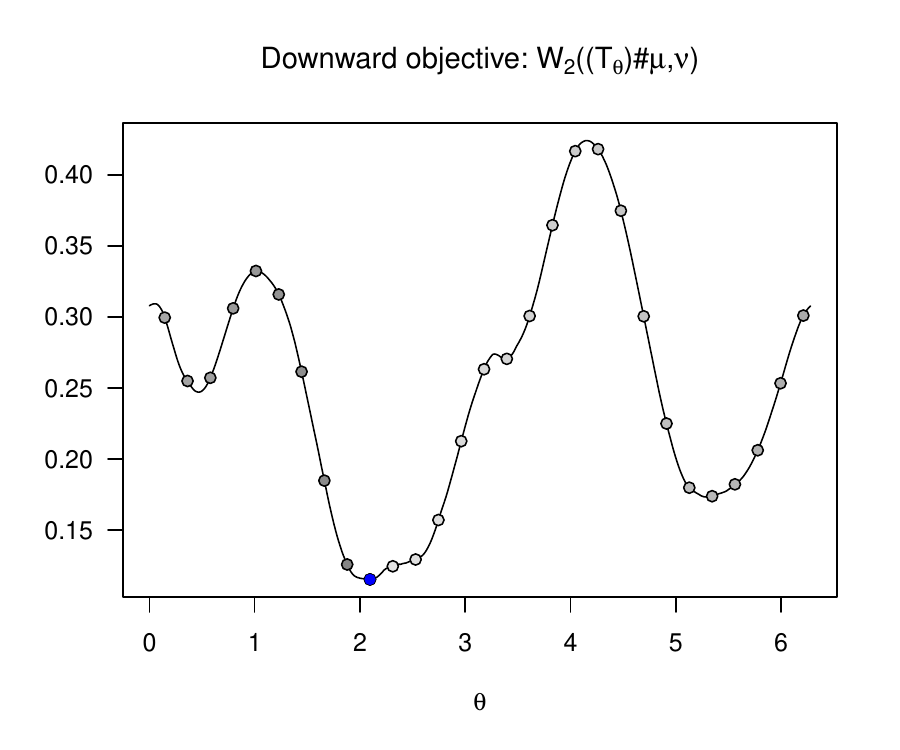}
	\caption{Empirical illustration of the Euclidean case and \cref{prop:Euclidean.equivalence}.}
	\label{fig:illustration}
\end{figure}

To illustrate the Euclidean case and \cref{prop:Euclidean.equivalence} we provide a simulated example in \cref{fig:illustration}. Here $n = 2$ and $d = 1$, and we parameterize $A \in \Hplanes$ by $A = \begin{bmatrix} \cos \theta & \sin \theta \end{bmatrix}^{\top}$ where $\theta \in [0, 2 \pi)$ represents the rotation angle. Each of $\mu \in \Prob_2(\R^2)$ and $\nu \in \Prob_2(\R)$ is a simulated point cloud with $300$ points, and has been normalized so that \cref{eqn:Euclidean.measure.normalization} holds. We choose $\nu$ to be a (normalized) empirical distribution sampled from $N(0, 1)$, and $\mu$ to be the empirical distribution of the tilted V-shaped point cloud in $\R^2$ in the left panel. The upward problem $\wass_2^{\uparrow}(\mu, \nu)$ is solved by finding an optimal $A_* \in \Hplanes$ which minimizes the $2$-Wasserstein distance between $\mu$ and $A_{\#} \nu$ on $\R^2$. In the left panel, the optimized $(A_*)_{\#} \nu$ is represented by the red points labeled by $+$, and the black line gives the column space of $A$. Equivalently, we may solve the downward problem $\wass_2^{\downarrow}(\mu, \nu)$ by minimizing the $2$-Wasserstein distance between $A_{\#}^{\top} \mu$ and $\nu$ on $\R$, and obtain the same $A_*$. In the right panel, we plot the density estimates of the optimized $(A_*)^{\top}_{\#} \mu$ (blue, dashed) and $\nu$ (red, solid). The thin curves (grey) are density estimates of $A_{\#}^{\top} \mu$ for several other (suboptimal) values of $A$. In the bottom panel, we plot the graph of $A \mapsto \wass_2(A_{\#}^{\top} \mu, \nu)$ as a function of the rotation angle $\theta \in [0, 2\pi)$. The dots indicate the pushforwards shown by the grey curves in the second graph. From the figure, we see clearly that the Wasserstein alignment problem is generally nonconvex and may posses multiple local minima. This example shows that uniqueness and stability of the solution may not hold even when $\Theta$ is one-dimensional.

\section{Proof of \texorpdfstring{\cref{prop:psiclaim}}{Proposition 4.2}} \label{sec:appendix.2}
Before giving the proof, we note that the claim is nontrivial since $\psi^c$ is itself a $c$-concave function. The strict convexity of $\psi^c$ implies that the Brenier map $T$ transporting $(\lambda)^\top_\#\mu$ to $\nu$ is a contraction, i.e.~$|T(y_2) - T(y_1)| < |y_2 - y_1|$ for $y_1 < y_2$. To show this, write $\psi^c(y) = \frac{1}{2}y^2 - \phi(y)$. Then $\phi$ is convex and is the Brenier potential from $\lambda_{\#}^{\top} \mu$ to $\nu$. The Brenier map is given by
\begin{equation} \label{eqn:explicit.example.Brenier}
	T(y) = \phi'(y) = y - (\psi^c)'(y),
\end{equation}
which is strictly increasing: for $y_1 < y_2$ we have $y_1 - (\psi^c)'(y_1) < y_2 - (\psi^c)'(y_2)$. Equivalently, $0 < (\psi^c)'(y_2) - (\psi^c)'(y_1) < y_2 - y_1$, where the first inequality follows from the strict convexity of $\psi^c$. It follows that $T(y_2) - T(y_1) < y_2 - y_1$.

To show that $\psi^c$ is strictly convex, it suffices to show that its derivative $(\psi^c)'$ is strictly increasing. From \cref{eqn:explicit.example.Brenier}, $(\psi^c)'$ is the identity minus the Brenier map transporting $\lambda_{\#}^\top \mu$ to $N(0,1)$. Let $c=\lambda^\top a$. Then we are looking for the Brenier map transporting the mixture $\frac{1}{2}N(c,1) + \frac{1}{2} N(-c,1)$ to $N(0,1)$. The solution to this is the function 
\begin{equation}\label{eq:brenierexample}
	\Phi^{-1}\left( \frac{1}{2} \Phi(\cdot - c) + \frac{1}{2} \Phi(\cdot + c) \right),
\end{equation}
where $\Phi$ is the standard normal cumulative distribution function.

Hence the claim boils down to showing that, for any $c\in \rr \setminus \{0\}$, the following function is strictly increasing on $\R$:
\begin{equation}\label{eq:whatisF}
	F(t) := t - \Phi^{-1}\left( \frac{1}{2} \Phi(t - c) + \frac{1}{2} \Phi(t + c) \right), \quad t \in \R.
\end{equation}
Let us argue that this is the case for $c=1$. For other values of $c$ the argument is similar. The plot of this function is given in \cref{fig:plotfninc}.
\begin{figure}[t!]
	\centering
	\includegraphics[width=0.55\linewidth]{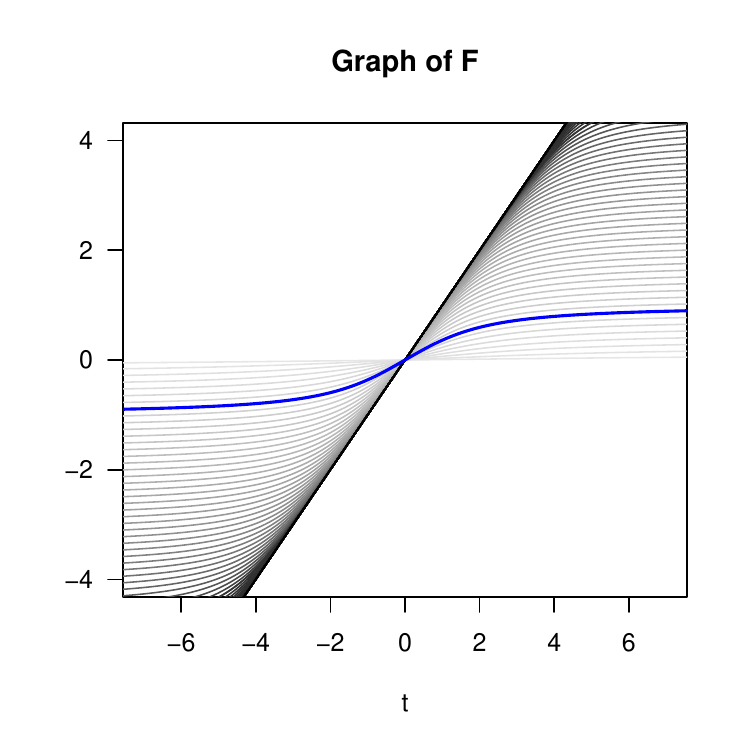}
	\vspace{-0.4cm}
	\caption{Graphs of the function $F(t)$ from \eqref{eq:whatisF} as $c$ increases from $0$ (light grey) to $\infty$ (black). The thick blue curve corresponds to $c = 1$.}
	\label{fig:plotfninc}
\end{figure}

Take $c=1$ and let $h(t):=\Phi^{-1}\left( \frac{1}{2} \Phi(t - 1) + \frac{1}{2} \Phi(t + 1) \right)$. Then $F(t)= t - h(t)$. We will prove that $F'(t) >0$ which is equivalent to proving $h'(t) < 1$. Let $\phi=\Phi'$ denote the standard normal density function. Then
\[
h'(t) = \frac{1}{\phi(h(t))}\left( \frac{1}{2} \phi(t - 1) + \frac{1}{2} \phi(t + 1) \right).
\]
Thus, to show that $h'(t) <1$, it is equivalent to proving that
\begin{equation}\label{eq:pfconvex}
	\phi(h(t)) > \frac{1}{2} \phi(t - 1) + \frac{1}{2} \phi(t + 1),
\end{equation}
where $h(t)$ is given by the condition 
\begin{equation}\label{eq:pfconvex2}
	\Phi(h(t))=\frac{1}{2} \Phi(t - 1) + \frac{1}{2} \Phi(t + 1).
\end{equation}

Since $\Phi$ is strictly increasing and is continuous, it is obvious that $t-1 < h(t) < t+1$. To prove \eqref{eq:pfconvex}, rewrite \eqref{eq:pfconvex2} as 
\begin{equation}\label{eq:pfconvex3}
	\int_{t-1}^{h(t)} \phi(u)du = \int_{h(t)}^{t+1} \phi(u)du. 
\end{equation}
Now, $\phi(u)$ is an injective function of $u$ when restricted to either $(0, \infty)$ or to $(-\infty, 0)$. We will consider several cases.     

\textbf{Case 1.} Suppose $t-1\ge 0$. Then $u \ge t-1 \ge 0$ and $v :=\phi(u)$ is a strictly decreasing function of $u$. By a change of variable, \eqref{eq:pfconvex3} can be written as
\begin{equation}\label{eq:pfconvex4}
	\int^{\phi(t-1)}_{\phi(h(t))} v \abs{\frac{du}{dv}} dv = \int_{\phi(t+1)}^{\phi(h(t))} v \abs{\frac{du}{dv}} dv.  
\end{equation}
Note that 
$v\abs{\frac{dv}{du}}= \frac{1}{u} \frac{1}{\sqrt{2\pi}}e^{-u^2/2}$,
which is strictly positive and strictly decreasing on $(0, \infty)$ as a function of $u$, and thus, strictly increasing as a function of $v$. Call this function $\chi(v)$. Thus, from \eqref{eq:pfconvex4}, 
\begin{equation}
	\begin{split}
		\chi(h(t))\left( \phi(t-1) - \phi(h(t)) \right) &< \int^{\phi(t-1)}_{\phi(h(t))} v \abs{\frac{du}{dv}} dv\\
		&= \int_{\phi(t+1)}^{\phi(h(t))} v \abs{\frac{du}{dv}} dv < \chi(h(t))\left( \phi(h(t)) - \phi(t+1) \right).
	\end{split}
\end{equation}
Canceling $\chi(h(t))$ from both sides gives us $\phi(t-1) - \phi(h(t)) < \phi(h(t)) - \phi(t+1)$, which is \eqref{eq:pfconvex}.

\textbf{Case 2.} Suppose $t+1 < 0$. This case is basically similar to the previous one. Now $v=\phi(u)$ is an increasing function of $u$ and $v \abs{\frac{dv}{du}}$ is strictly increasing in $v$. Thus,
\begin{equation}
	\begin{split}
		\chi(h(t))\left( \phi(h(t)) - \phi(t-1) \right) &> \int_{\phi(t-1)}^{\phi(h(t))} v \abs{\frac{du}{dv}} dv\\
		&= \int^{\phi(t+1)}_{\phi(h(t))} v \abs{\frac{du}{dv}} dv > \chi(h(t))\left( \phi(t+1) - \phi(h(t)) \right).
	\end{split}
\end{equation}
This leads to the same conclusion. 

\textbf{Case 3.} Suppose $t-1 < 0  < t+1$. In this case we divide \eqref{eq:pfconvex3} into integrals for $\{u >0\}$ and integrals for $\{u<0\}$. For the former, we apply the logic of Case 1, and for the latter we apply the logic of Case 2. Combining both parts give us \eqref{eq:pfconvex}. We skip the details.

\section{Implementation} \label{sec:implementation}
\subsection{An LP formulation for empirical measures} \label{sec:primal.dual}

Let us consider a discrete setup and rewrite the dual problem in \cref{thm:duality} in a standard linear programming (LP) format. Suppose that $\mu$ is an empirical distribution of the form $\sum_{i=1}^N p_i \delta_{x_i}$ and $\nu$ is similarly given by $\sum_{j=1}^M q_j \delta_{z_j}$. Here $(p_1, \ldots, p_N)$ and $(q_1, \ldots, q_M)$ are both probability vectors. Also assume that $\Theta$ is a finite set given by $\{\theta_1, \ldots, \theta_l\}$. We can now reduce all functions to vectors in Euclidean spaces:
\[
\xi_{ik}:= \xi(x_i, \theta_k), \qquad \psi_j:=\psi(z_j),\qquad c_{ijk}:=c(T_{\theta_k}x_i, z_j).
\]
It is straightforward to include the penalization function when it is present.

For any positive integer $j$, let $[j]$ denote the set $\{1,2,\ldots, j\}$. The dual constraint \cref{eq:dualconstraint} now reads
\begin{equation}\label{eq:dualconstraintdiscrete}
	\xi_{ik} + \psi_j \le c_{ijk},\quad \forall \; (i,j,k) \in [N]\times [M]\times [l],
\end{equation}
while the constraint that $\xi \in \mathcal{F}_\mu$ now reads
\begin{equation}\label{eq:Fmudiscrete}
	\sum_{i=1}^N \xi_{ik} p_i = \sum_{i=1}^N \xi_{i1} p_i,\; \forall \; k\in [l].  
\end{equation}
Hence the dual problem in \cref{eq:cxdual} can be written as
\begin{equation} \label{eqn:LP.discrete}
	\max\left\{ \sum_{i=1}^N \xi_{i1} p_i + \sum_{j=1}^M \psi_j q_j \right\},
\end{equation}
where the maximum is taken over the free variables
\[
\xi_{ik}, \psi_j \in \rr, \quad \forall i \in [N],\; j \in [M],\; k \in [l],  
\]
satisfying linear constraints \cref{eq:dualconstraintdiscrete} and \cref{eq:Fmudiscrete} where the constants $(c_{ijk})$ are fixed. This finite-dimensional LP problem can now be solved by any standard LP solver. 

Furthermore, once an optimal pair $(\psi, \xi)$ has been found, an optimal $\theta_* \in \{\theta_1,\ldots, \theta_l\}$ can also be easily found via the following observation. For each $k\in [l]$, consider the vector of slack variables 
\[
c_{ijk} - \xi_{ik} - \psi_j,\quad (i,j)\in [N] \times [M].
\]
By our duality, there will exist some $k\in [l]$, for which the usual Kantorovich duality will hold and $\xi_{ik}=\psi^{\overline{c}}_{ik}$, where $\psi^{\overline{c}}_{ik} := \min_{j=1,\ldots, M}\{c_{ijk} - \psi_j\}$.
This characterizes any optimal $\theta_*$, via complementary slackness. 

One may also use \cref{cor:optimality} directly. The solver outputs a solution to the linear programming problem \cref{eqn:LP.discrete} for which strong duality holds. Hence, the duality gap in \cref{eqn:gap2} is zero and \cref{cor:optimality} applies. Hence
\begin{equation}\label{eq:theta_*}
	\theta_* \in \argmin_{k \in [l]} I_{\theta_k}, \quad \mbox{where} \quad I_{\theta_k} := \sum_{i=1}^N p_i \psi^{\overline{c}}_{ik}. 
\end{equation}

\paragraph{Continuous transformation families}
The finite-$\Theta$ LP above also suggests a possible route to extend to the case where $\Theta$ is continuous. 
Assume that $\Theta$ is a compact subset of $\R^d$, while $\mu$ and $\nu$ remain empirical. Introducing a scalar \(m\) for the common value of $\int \psi(x, \theta) d \mu(x) = \sum_{i=1}^N p_i \xi_i(\theta)$ across $\theta$ (when $\xi \in \mathcal{F}_{\mu}$), the dual problem in \cref{thm:duality2} can be rewritten as the semi-infinite program
\begin{equation}\label{eq:semiinfinite.master}
	\sup_{m\in \mathbb{R},\, \psi\in \mathbb{R}^M}
	\left\{
	m+\sum_{j=1}^M q_j\psi_j \;:\; m\le I_\theta(\psi)\quad \forall \theta\in\Theta
	\right\},
\end{equation}
where
\[
I_\theta(\psi):=\sum_{i=1}^N p_i (\psi^{\overline c_\theta})_i \quad \text{and} \quad (\psi^{\overline c_\theta})_i:= \min_{j\in[M]}\{c(T_\theta x_i,z_j)-\psi_j\},
\]
with the obvious modification when a penalization term is present. 
Indeed, for fixed \((m,\psi)\), the largest feasible value of \(\sum_{i=1}^N p_i \xi_i(\theta)\) under the constraints
\[
\xi_i(\theta)+\psi_j\le c(T_\theta x_i,z_j),\qquad i\in[N],\ j\in[M],
\]
is precisely \(I_\theta(\psi)\), attained by taking
\[
\xi_i(\theta)=\min_{j\in[M]}\{c(T_\theta x_i,z_j)-\psi_j\}.
\]
Hence the existence of \(\xi_i(\theta)\) satisfying
\[
\xi_i(\theta)+\psi_j\le c(T_\theta x_i,z_j),\qquad
m\le \sum_{i=1}^N p_i\xi_i(\theta),
\]
is equivalent to the single inequality \(m\le I_\theta(\psi)\).

Formulation \eqref{eq:semiinfinite.master} naturally leads to an exchange (or cutting-plane) scheme. 
Starting from a finite active set \(\Theta^{(r)}=\{\theta^1,\dots,\theta^r\}\subset \Theta\), one solves the restricted master LP
\[
\max_{m,\psi,\xi^1,\dots,\xi^r}
\left\{
m+\sum_{j=1}^M q_j\psi_j
\right\}
\]
subject to
\[
\xi_i^s+\psi_j\le c(T_{\theta^s}x_i,z_j),\qquad
m\le \sum_{i=1}^N p_i \xi_i^s,
\quad
\forall i\in[N],\ j\in[M],\ s\in[r].
\]
Let \((m^{(r)},\psi^{(r)},\xi^{1,(r)},\dots,\xi^{r,(r)})\) denote an optimal solution of this restricted problem.

One then considers the pricing (or separation) problem
\[
\min_{\theta\in\Theta} I_\theta(\psi^{(r)}),
\qquad
I_\theta(\psi)=\sum_{i=1}^N p_i \min_{j\in[M]}
\{c(T_\theta x_i,z_j)-\psi_j\}.
\]
For general continuous families this auxiliary problem need not be convex, and it may itself be the main computational difficulty. 
However, it is an optimization problem only in the transformation parameter \(\theta\), with the transport variables already eliminated. 
When \(c\) and \(\theta\mapsto T_\theta x_i\) are smooth, \(I_\theta(\psi)\) is a finite lower envelope of smooth functions and is therefore continuous and piecewise smooth, although typically nonconvex and nonsmooth at switching points. 
This makes adaptive gridding, multistart local search, subgradient-type methods, and branch-and-bound natural candidates, especially when \(\dim(\Theta)\) is small.

To continue the exchange algorithm, it suffices to find any \(\hat\theta\in\Theta\) such that
\[
I_{\hat\theta}(\psi^{(r)})<m^{(r)},
\]
since this yields a violated constraint that can be added to the active set. 
By contrast, a certified stopping criterion requires solving the pricing problem globally, or at least obtaining a rigorous lower bound for it. 
In particular, if
\[
\underline I^{(r)}:=\min_{\theta\in\Theta} I_\theta(\psi^{(r)}) \ge m^{(r)}-\varepsilon,
\]
then the current restricted solution is \(\varepsilon\)-optimal for the full semi-infinite problem. 
Moreover, \((\underline I^{(r)},\psi^{(r)})\) is feasible for \eqref{eq:semiinfinite.master}, so
\[
m^{(r)}-\underline I^{(r)}
\]
yields a computable optimality gap whenever the pricing problem is solved globally.

This viewpoint may also be interpreted as the dual counterpart of column generation: each newly generated \(\theta\) contributes one additional transformation block to the restricted LP. 
Thus, the dual formulation does not remove the computational difficulty associated with optimizing over a continuous transformation family, but it does isolate that difficulty in a lower-dimensional pricing oracle, while the master problem remains a standard OT-type LP on the empirical supports. 
In this sense, the conceptual advantage of the dual viewpoint is preserved: one explores the transformation family adaptively, rather than through an a priori fine discretization of \(\Theta\). We also note that in specific alignment problems there may be additional structures that can be exploited.

\medskip

Let us now illustrate the effectiveness of our approach on a few synthetic data sets.

\subsection{Alignment of two 2D point clouds}
Shape analysis is a fundamental field in statistics, computer vision, and computational geometry that studies the properties of geometric shapes and their transformations~\cite{Dryden-Mardia-2016,  Srivastava-2016}. It involves tasks such as shape matching, classification, recognition, registration, and deformation modeling. Applications span various domains, including medical imaging, robotics, object recognition, 2D/3D reconstruction, and biometric identification.

A key challenge in shape analysis is the alignment of two 2D point clouds, a process known as \textit{shape registration}. The \textit{Iterative Closest Point (ICP)} algorithm~\cite{besl1992method,chen1992object} is a point cloud registration algorithm that is a widely used method for rigid shape alignment. Although ICP is computationally efficient and effective for rigid transformations, it is sensitive to initialization and noise, often leading to suboptimal results~\cite{Srivastava-2016}.

\begin{figure}
	\centering
	\includegraphics[scale=0.33]{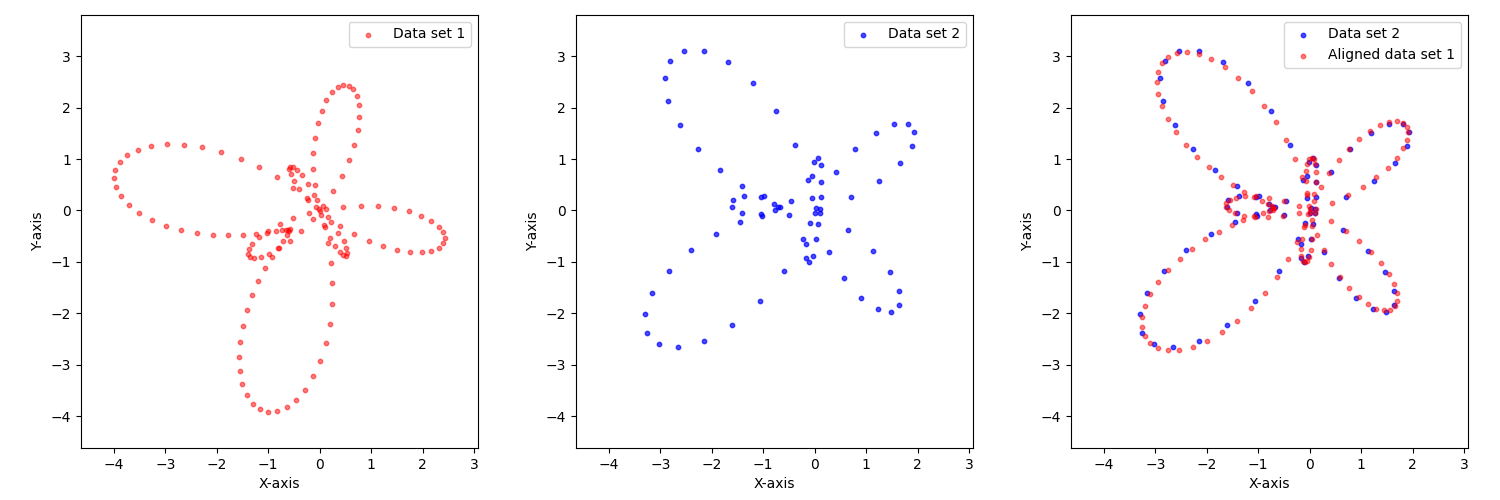}
	\includegraphics[scale=0.33]{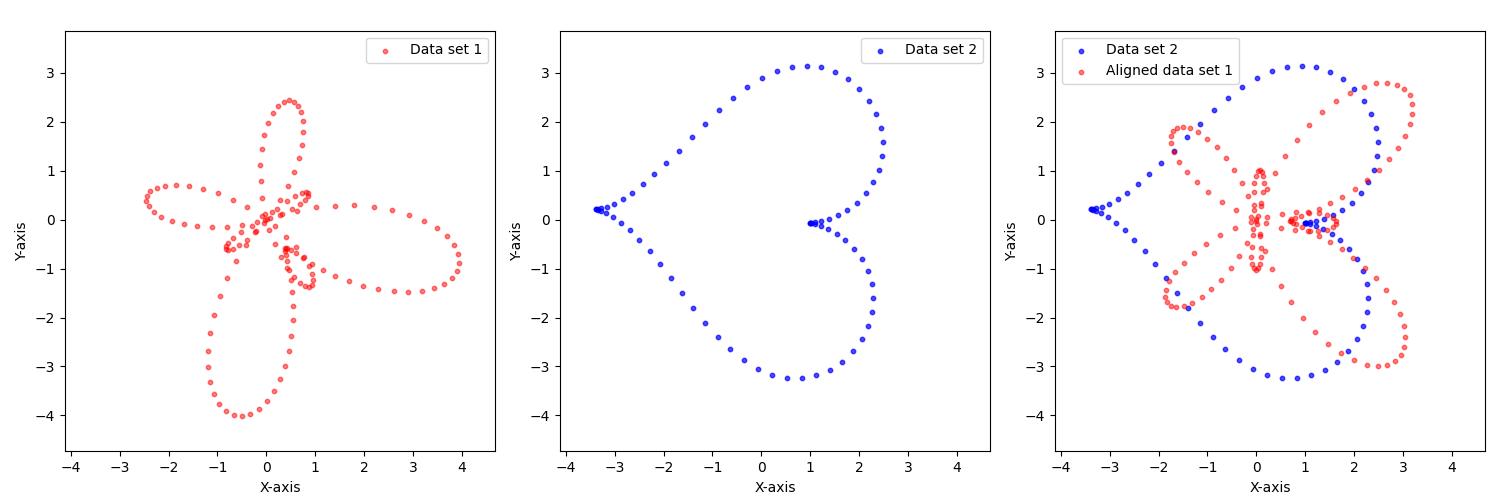}
	\includegraphics[scale=0.33]{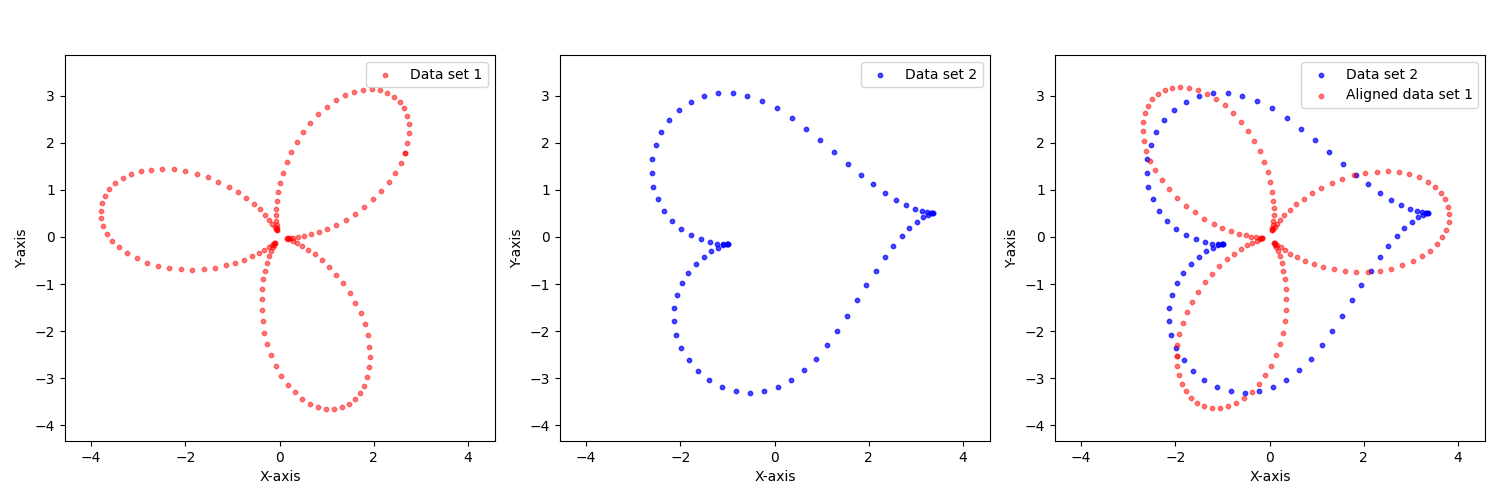}
	\caption{{\bf Left}: The discrete distribution $\mu$ with $N = 150$ equally weighted atoms. {\bf Center}: The discrete distribution $\nu$ with $M= 80$ equally weighted atoms. {\bf Right}: The two distributions optimally aligned.}
	\label{fig:Shapes}
\end{figure}

Our proposed Wasserstein alignment method with the convex Kantorovich-type dual provides an effective solution for aligning two sets of data points. Unlike traditional approaches, it offers greater flexibility and robustness, particularly in scenarios involving partial correspondences and non-rigid deformations. Moreover, being a convex optimization problem, it avoids the limitations of the ICP algorithm, which is often sensitive to initialization and prone to local minima. To demonstrate the effectiveness of our approach, we present three toy examples in shape analysis, following the spirit of~\cite[Chapter 2]{Srivastava-2016}.

For our illustrations, we utilize the linear program described in \cref{sec:primal.dual}, implemented using the off-the-shelf \texttt{linprog} solver in {\it Python}. We consider three distinct shapes: (i) {\it butterfly}, (ii) {\it heart}, and (iii) {\it flower}. Each row in \cref{fig:Shapes} presents the Wasserstein alignment between two datasets, where each dataset is generated by applying a random 2D rotation to one of the three shapes.

In each example, the two discrete distributions, $\mu$ and $\nu$, are supported on $N=150$ and $M=80$ points, respectively. We consider the rotation space $\Theta=[0,2\pi)$, where each transformation $T_\theta$ is anticlockwise rotation by angle $\theta$ in $\R^2$. For our implementation, we use an equi-spaced grid of angles $\{\theta_1,\ldots, \theta_l\}$ with $l = 40$. Each row in \cref{fig:Shapes} is structured as follows:
\begin{itemize}
	\item The first two plots show the two unaligned data sets, corresponding to the different shapes.
	
	\item The third plot displays the optimal Wasserstein alignment of the data sets
\end{itemize} 
Notably, the optimal alignments in the second and third examples of \cref{fig:Shapes} preserve the inherent symmetries of the underlying shapes. Moreover, the results align closely with our intuitive expectations of how these data sets should be matched, further demonstrating the effectiveness of our approach.

\section*{Acknowledgments}
This project started during our visit to The Mathematics of Data workshop at the Institute for Mathematical Sciences (IMS), National University of Singapore in 2024. We thank the IMS and the organizers Afonso Bandera, Subhro Ghosh, and Philippe Rigollet for inviting us. We also thank Zhengxin Zhang for pointers to the Gromov--Wasserstein literature and Adam Jaffe for helpful conversations.

\bibliographystyle{siamplain}
\bibliography{references}

\begin{thebibliography}{10}

\bibitem{adve25}
{\sc A.~Adve and A.~R. M{\'e}sz{\'a}ros}, {\em On nonexpansiveness of metric
  projection operators on {W}asserstein spaces}, Advances in Calculus of
  Variations, 18 (2025), pp.~1207--1222.

\bibitem{ACJ20}
{\sc A.~Alfonsi, J.~Corbetta, and B.~Jourdain}, {\em Sampling of probability
  measures in the convex order by {W}asserstein projection}, Annales de
  l'Institut Henri Poincar{\'e} (B) Probabilit{\'e}s et Statistiques, 56
  (2020), pp.~1706--1729.

\bibitem{Alvarez2018}
{\sc D.~Alvarez-Melis and T.~Jaakkola}, {\em {G}romov-{W}asserstein alignment
  of word embedding spaces}, in Proceedings of the 2018 conference on empirical
  methods in natural language processing, 2018, pp.~1881--1890.

\bibitem{Alvarez2019}
{\sc D.~Alvarez-Melis, S.~Jegelka, and T.~S. Jaakkola}, {\em Towards optimal
  transport with global invariances}, in The 22nd International Conference on
  Artificial Intelligence and Statistics, PMLR, 2019, pp.~1870--1879.

\bibitem{AGS05}
{\sc L.~Ambrosio, N.~Gigli, and G.~Savar\'e}, {\em Gradient Flows: In Metric
  Spaces and in the Space of Probability Measures}, Birkh{\"a}user, second~ed.,
  2008.

\bibitem{besl1992method}
{\sc P.~J. Besl and N.~D. McKay}, {\em Method for registration of 3-d shapes},
  in Sensor Fusion IV: Control Paradigms and Data Structures, vol.~1611, Spie,
  1992, pp.~586--606.

\bibitem{Blumberg-2020}
{\sc A.~J. Blumberg, M.~Carriere, M.~A. Mandell, R.~Rabadan, and S.~Villar},
  {\em {MREC}: a fast and versatile framework for aligning and matching point
  clouds with applications to single cell molecular data}, arXiv preprint
  arXiv:2001.01666,  (2020).

\bibitem{BRPP15}
{\sc N.~Bonneel, J.~Rabin, G.~Peyr{\'e}, and H.~Pfister}, {\em Sliced and
  {R}adon {W}asserstein barycenters of measures}, Journal of Mathematical
  Imaging and Vision, 51 (2015), pp.~22--45.

\bibitem{Bousquet2017}
{\sc O.~Bousquet, S.~Gelly, I.~Tolstikhin, C.-J. Simon-Gabriel, and
  B.~Schoelkopf}, {\em From optimal transport to generative modeling: the vegan
  cookbook}, arXiv preprint arXiv:1705.07642,  (2017).

\bibitem{Bryant-2022}
{\sc P.~Bryant, G.~Pozzati, and A.~Elofsson}, {\em Improved prediction of
  protein-protein interactions using alphafold2}, Nature communications, 13
  (2022), p.~1265.

\bibitem{Bunne-2019}
{\sc C.~Bunne, D.~Alvarez-Melis, A.~Krause, and S.~Jegelka}, {\em Learning
  generative models across incomparable spaces}, in International conference on
  machine learning, PMLR, 2019, pp.~851--861.

\bibitem{Bunne2024}
{\sc C.~Bunne, G.~Schiebinger, A.~Krause, A.~Regev, and M.~Cuturi}, {\em
  Optimal transport for single-cell and spatial omics}, Nature Reviews Methods
  Primers, 4 (2024), p.~58.

\bibitem{CWZ25}
{\sc Z.~Cang, Y.~Wu, and Y.~Zhao}, {\em Supervised {G}romov--{W}asserstein
  optimal transport with metric-preserving constraints}, SIAM Journal on
  Mathematics of Data Science, 7 (2025), pp.~301--328.

\bibitem{CG03}
{\sc E.~A. Carlen and W.~Gangbo}, {\em Constrained steepest descent in the
  2-{W}asserstein metric}, Annals of Mathematics,  (2003), pp.~807--846.

\bibitem{chen1992object}
{\sc Y.~Chen and G.~Medioni}, {\em Object modelling by registration of multiple
  range images}, Image and vision computing, 10 (1992), pp.~145--155.

\bibitem{CMP17}
{\sc P.-A. Chiappori, R.~J. McCann, and B.~Pass}, {\em Multi-to one-dimensional
  optimal transport}, Communications on Pure and Applied Mathematics, 70
  (2017), pp.~2405--2444.

\bibitem{Courty2016}
{\sc N.~Courty, R.~Flamary, D.~Tuia, and A.~Rakotomamonjy}, {\em Optimal
  transport for domain adaptation}, IEEE transactions on pattern analysis and
  machine intelligence, 39 (2016), pp.~1853--1865.

\bibitem{Demetci2022}
{\sc P.~Demetci, R.~Santorella, B.~Sandstede, W.~S. Noble, and R.~Singh}, {\em
  {SCOT}: single-cell multi-omics alignment with optimal transport}, Journal of
  Computational Biology, 29 (2022), pp.~3--18.

\bibitem{Dryden-Mardia-2016}
{\sc I.~L. Dryden and K.~V. Mardia}, {\em Statistical Shape Analysis with
  Applications in {R}}.

\bibitem{Fuchs-2020}
{\sc F.~Fuchs, D.~Worrall, V.~Fischer, and M.~Welling}, {\em
  S{E}(3)-transformers: 3d roto-translation equivariant attention networks},
  Advances in neural information processing systems, 33 (2020), pp.~1970--1981.

\bibitem{Goodall-1991}
{\sc C.~Goodall}, {\em Procrustes methods in the statistical analysis of
  shape}, Journal of the Royal Statistical Society: Series B (Methodological),
  53 (1991), pp.~285--321.

\bibitem{Gower-Dijksterhuis-2004}
{\sc J.~C. Gower and G.~B. Dijksterhuis}, {\em Procrustes Problems}, Oxford
  University Press, 2004.

\bibitem{Grave2019}
{\sc E.~Grave, A.~Joulin, and Q.~Berthet}, {\em Unsupervised alignment of
  embeddings with {W}asserstein {P}rocrustes}, in The 22nd International
  Conference on Artificial Intelligence and Statistics, PMLR, 2019,
  pp.~1880--1890.

\bibitem{KW25}
{\sc K.~Kato and B.~Wang}, {\em Convergence of empirical {G}romov-{W}asserstein
  distance}, arXiv preprint arXiv:2508.03985,  (2025).

\bibitem{VAE}
{\sc D.~P. Kingma and M.~Welling}, {\em Auto-encoding variational {B}ayes},
  arXiv preprint arXiv:1312.6114,  (2013).

\bibitem{VAEintro}
{\sc D.~P. Kingma and M.~Welling}, {\em An introduction to variational
  autoencoders}, Foundations and Trends{\textregistered} in Machine Learning,
  12 (2019), pp.~307--392.

\bibitem{Koehl-2023}
{\sc P.~Koehl, M.~Delarue, and H.~Orland}, {\em Computing the
  {G}romov-{W}asserstein distance between two surface meshes using optimal
  transport}, Algorithms, 16 (2023), p.~131.

\bibitem{MP20}
{\sc R.~J. McCann and B.~Pass}, {\em Optimal transportation between unequal
  dimensions}, Archive for Rational Mechanics and Analysis, 238 (2020),
  pp.~1475--1520.

\bibitem{Memoli-2009}
{\sc F.~M{\'e}moli}, {\em Spectral {G}romov-{W}asserstein distances for shape
  matching}, in 2009 IEEE 12th International Conference on Computer Vision
  Workshops, ICCV Workshops, IEEE, 2009, pp.~256--263.

\bibitem{Memoli-2011}
{\sc F.~M{\'e}moli}, {\em {G}romov--{W}asserstein distances and the metric
  approach to object matching}, Foundations of Computational Mathematics, 11
  (2011), pp.~417--487.

\bibitem{Mikolov-2013}
{\sc T.~Mikolov, Q.~V. Le, and I.~Sutskever}, {\em Exploiting similarities
  among languages for machine translation}, arXiv preprint arXiv:1309.4168,
  (2013).

\bibitem{Nguyen-2024}
{\sc D.~M. Nguyen, N.~Lukashina, T.~Nguyen, A.~T. Le, T.~Nguyen, N.~Ho,
  J.~Peters, D.~Sonntag, V.~Zaverkin, and M.~Niepert}, {\em Structure-aware e
  (3)-invariant molecular conformer aggregation networks}, arXiv preprint
  arXiv:2402.01975,  (2024).

\bibitem{Peyre2019}
{\sc G.~Peyr{\'e}, M.~Cuturi, et~al.}, {\em Computational optimal transport:
  With applications to data science}, Foundations and Trends{\textregistered}
  in Machine Learning, 11 (2019), pp.~355--607.

\bibitem{pomerleau2013}
{\sc F.~Pomerleau, F.~Colas, R.~Siegwart, and S.~Magnenat}, {\em Comparing
  {ICP} variants on real-world data sets: Open-source library and experimental
  protocol}, Autonomous robots, 34 (2013), pp.~133--148.

\bibitem{rusinkiewicz2001}
{\sc S.~Rusinkiewicz and M.~Levoy}, {\em Efficient variants of the {ICP}
  algorithm}, in Proceedings Third International Conference on 3-D Digital
  Imaging and Modeling, IEEE, 2001, pp.~145--152.

\bibitem{santam2015ot}
{\sc F.~Santambrogio}, {\em Optimal Transport for Applied Mathematicians:
  Calculus of Variations, PDEs, and Modeling}, Springer, 2015.

\bibitem{Schonemann-1966}
{\sc P.~H. Sch\"onemann}, {\em A generalized solution of the orthogonal
  {P}rocrustes problem}, Psychometrika, 31 (1966), pp.~1--10.

\bibitem{Srivastava-2016}
{\sc A.~Srivastava and E.~P. Klassen}, {\em Functional and Shape Data
  Analysis}, Springer Series in Statistics, Springer-Verlag, New York, 2016.

\bibitem{Tolstikhin2017}
{\sc I.~Tolstikhin, O.~Bousquet, S.~Gelly, and B.~Schoelkopf}, {\em Wasserstein
  {A}uto-{E}ncoders}, in International Conference on Learning Representations,
  2018, \url{https://openreview.net/forum?id=HkL7n1-0b}.

\bibitem{Villani2003}
{\sc C.~Villani}, {\em Topics in Optimal Transportation}, American Mathematical
  Society, 2003.

\bibitem{V08}
{\sc C.~Villani}, {\em Optimal Transport: Old and New}, Springer, 2008.

\bibitem{Xing-2015}
{\sc C.~Xing, D.~Wang, C.~Liu, and Y.~Lin}, {\em Normalized word embedding and
  orthogonal transform for bilingual word translation}, in Proceedings of the
  2015 conference of the North American chapter of the association for
  computational linguistics: human language technologies, 2015, pp.~1006--1011.

\bibitem{Xu-2019}
{\sc H.~Xu, D.~Luo, H.~Zha, and L.~C. Duke}, {\em Gromov-{W}asserstein learning
  for graph matching and node embedding}, in International conference on
  machine learning, PMLR, 2019, pp.~6932--6941.

\bibitem{Yan-2018}
{\sc Y.~Yan, W.~Li, H.~Wu, H.~Min, M.~Tan, and Q.~Wu}, {\em Semi-supervised
  optimal transport for heterogeneous domain adaptation.}, in IJCAI, vol.~7,
  2018, pp.~2969--2975.

\bibitem{Zhang-2024}
{\sc Z.~Zhang, Z.~Goldfeld, K.~Greenewald, Y.~Mroueh, and B.~K. Sriperumbudur},
  {\em Gradient flows and {R}iemannian structure in the {G}romov-{W}asserstein
  geometry}, Foundations of Computational Mathematics, 26 (2025),
  pp.~1911--–2003.

\bibitem{ZGMS24}
{\sc Z.~Zhang, Z.~Goldfeld, Y.~Mroueh, and B.~K. Sriperumbudur}, {\em
  {G}romov--{W}asserstein distances: Entropic regularization, duality and
  sample complexity}, The Annals of Statistics, 52 (2024), pp.~1616--1645.

\end{thebibliography}

\end{document}